\documentclass[pdflatex,sn-mathphys-num]{sn-jnl}

\usepackage{graphicx}%
\usepackage{multirow}%
\usepackage{amsmath,amssymb,amsfonts}%
\usepackage{amsthm}%
\usepackage{mathrsfs}%
\usepackage[title]{appendix}%
\usepackage{xcolor}%
\usepackage{textcomp}%
\usepackage{manyfoot}%
\usepackage{booktabs}%
\usepackage{algorithm}%
\usepackage{algorithmicx}%
\usepackage{algpseudocode}%
\usepackage{listings}%

\theoremstyle{thmstyleone}%
\newtheorem{theorem}{Theorem}
\newtheorem{proposition}[theorem]{Proposition}%

\theoremstyle{thmstyletwo}%
\newtheorem{example}{Example}%
\newtheorem{lemma}{Lemma}%
\newtheorem{corollary}{Corollary}

\theoremstyle{thmstylethree}%
\newtheorem{definition}{Definition}%

\begin{document}

\title[A Note on Partially Anti-invariant Submanifolds of Kenmotsu Manifolds]{A Note on Partially Anti-invariant Submanifolds of Kenmotsu Manifolds}

\author*[1]{\fnm{Mohammad} \sur{Shuaib}}\email{shuaibyousuf6@gmail.com}

\author[2]{\fnm{Cenap} \sur{Ozel}}\email{cenap.ozel@gmail.com}
\equalcont{These authors contributed equally to this work.}

\author[3]{\fnm{Meraj} \sur{A. Khan}}\email{meraj79@gmail.com}
\equalcont{These authors contributed equally to this work.}

\author[4]{\fnm{Yakup} \sur{Yildirim}}\email{yyildirim@biruni.edu.tr}
\equalcont{These authors contributed equally to this work.}

\affil*[1]{\orgdiv{Department of Mathematics}, \orgname{Lovely Professional University}, \orgaddress{\street{Phagwara}, \postcode{144411}, \state{Punjab}, \country{India}}}

\affil[2]{\orgdiv{Department of Mathematics}, \orgname{King Abdudlaziz University}, \orgaddress{\state{Jeddah}, \country{K.S.A}}}

\affil[3]{\orgdiv{Department of Mathematics}, \orgname{Imam Muhammad Ibn Saud Islamic University}, \orgaddress{\street{Riyadh}, \country{K.S.A}}}

\affil[4]{\orgdiv{Department of Computer Science and Engineering}, \orgname{Biruni University}, \postcode{34010}, \state{Istanbul}, \country{Turkey}}

\affil[4]{\orgdiv{Mathematics research Center}, \orgname{Near East University}, \postcode{99138}, \state{Nicosia}, \country{Cyprus}}


\abstract{In this paper, we introduce and investigate partially anti-invariant submanifolds of Kenmotsu manifolds. The integrability of the involved distributions is examined, and conditions ensuring that the induced foliations are totally geodesic are obtained. We further explore the geometric properties of these distributions when the associated structure tensors $N$ and $P$ are parallel. By making use of the curvature tensor of Kenmotsu manifolds and the properties of Kenmotsu space forms with constant $\phi$-sectional curvature, we derive various geometric characterizations of partially anti-invariant submanifolds. In addition, we characterize totally umbilical partially anti-invariant submanifolds and establish an inequality for partially anti-invariant submanifolds of Kenmotsu space forms, including the characterization of the equality case.
}




\maketitle

\section{Introduction}

The study of submanifolds in manifolds endowed with additional geometric structures has remained an important research area in differential geometry. The presence of an ambient tensor field imposes remarkable restrictions on the geometry of submanifolds and gives rise to various distinguished classes. In particular, submanifold theory in K\"ahler and almost contact metric manifolds has developed extensively due to the close relationship between the structure of the ambient space and the induced geometry on the submanifold.
In K\"ahler geometry, holomorphic and totally real submanifolds are regarded as two fundamental examples of submanifolds invariant and orthogonal to the almost complex structure, respectively. To establish a unified framework containing both these classes, Bejancu \cite{bej1} introduced CR-submanifolds by considering a tangent bundle decomposition into invariant and anti-invariant distributions. This concept provided a foundation for studying more general submanifolds possessing mixed geometric properties.

The analogous notions in almost contact metric geometry are invariant and anti-invariant submanifolds. A submanifold $M$ of an almost contact metric manifold is called invariant if the structure tensor $\phi$ preserves the tangent space of $M$, that is, $\phi(T_pM)\subseteq T_pM$, whereas it is said to be anti-invariant if  $\phi(T_pM)\subseteq T_p^\perp M$ for every $p\in M$. These classes were investigated by Yano and Kon \cite{kon76,yk76}. Later, the theory of CR-submanifolds was extended to the setting of almost contact metric manifolds by Yano and Kon \cite{kon1980}, which stimulated further research on mixed-type submanifolds in almost contact geometry \cite{munt}. Subsequently, Kenmotsu \cite{ken1} introduced a new class of almost contact metric manifolds, now known as Kenmotsu manifolds. Since then, the geometry of CR-submanifolds and skew CR-submanifolds in Kenmotsu manifolds has been extensively investigated by several authors \cite{shk1,naghi1}.

A significant development in submanifold theory was the introduction of slant submanifolds by Chen \cite{ch10} in K\"ahler manifolds. This notion generalized both holomorphic and totally real submanifolds by allowing the angle between the almost complex structure and the tangent space to remain constant. Subsequently, Lotta \cite{lot} extended the concept of slant submanifolds to almost contact metric manifolds. In this direction, the contributions of \c{S}ahin \cite{sah1,sah2,sah3} have played an important role in understanding the geometry of slant submanifolds in different ambient spaces, including K\"ahler, quaternion K\"ahler, and product manifolds.
Among the generalized classes of slant submanifolds, semi-slant submanifolds form an important category in almost contact metric geometry. Cabrerizo \textit{et al.} \cite{cab1} introduced semi-slant submanifolds of Sasakian manifolds and investigated their basic geometric properties. Further, many authors studied semi-slant submanifold of kenmotsu manifolds by geometrical point of view and also studied properties of them \cite{va1,siraj1}. Moreover, Khan and Shuaib  \cite{shk2,shk3} investigated the geometry of pseudo-slant submanifolds of Kenmotsu manifolds and obtained several fundamental geometric properties and characterization results.


In a subsequent development, Firat and Sahin \cite{firat} introduced PTR-submanifolds of Kaehler manifolds and investigated their fundamental geometric properties through a suitable decomposition of the tangent bundle. Motivated by Chen's pioneering concept of generic submanifolds together with the work of Firat and Sahin, we consider the corresponding problem in the setting of Kenmotsu manifolds. In contrast to K\"ahler manifolds, Kenmotsu manifolds are endowed with the characteristic vector field $\xi$, whose presence has a significant influence on the decomposition of the tangent bundle and the behaviour of the second fundamental form. Taking these distinctive features into account, we introduce the notion of partially anti-invariant submanifolds of Kenmotsu manifolds, in which one distinguished tangent distribution is anti-invariant, while its complementary distribution is allowed to possess a more general geometric structure.


The main objective of this paper is to investigate the geometric properties of partially anti-invariant submanifolds of Kenmotsu manifolds.
The organization of the paper is as follows. Section 2 recalls some fundamental concepts and results related to almost contact metric manifolds specially of Kenmotsu manifolds and introduces partially anti-invariant submanifolds. The integrability conditions of the associated distributions and the geometry of their foliations are also discussed. Section 3 deals with partially anti-invariant submanifolds satisfying certain parallelism conditions on the structure tensors, where several geometric consequences are obtained. In Section 4, we investigate curvature-related properties of partially anti-invariant submanifolds in Kenmotsu space forms. The final section is devoted to totally umbilical partially anti-invariant submanifolds and the derivation of an inequality.

\section{Preliminaries}
For the sake of completeness, we recall some basic notions, identities, and conventions from almost contact metric manifold that will be used throughout the paper.

Let $\bar M$ be a $(2n+1)$ dimensional almost contact manifold with almost contact
structures $(\phi,\xi,\eta)$, where $\phi$ is a $(1, 1)$ tensor field, $\xi$ a vector field and $\eta$, a 1-form satisfying
\begin{equation}\label{e1}
	\phi^{2}=-I+\eta\otimes\xi,\qquad
	\eta(\xi)=1,\qquad
	\phi\xi=0,\qquad
	\eta\circ\phi=0.
\end{equation}
An almost contact manifold admits a Riemannian metric $g$ compatible with the almost contact structure $(\phi,\xi,\eta)$ such that
\begin{equation}\label{e2}
	g(\phi {X_1},\phi {X_2})=g({X_1},{X_2})-\eta({X_1})\eta({X_2}),
\end{equation}
for all vector fields ${X_1},{X_2}\in\Gamma(T\bar M)$, then the structure
$(\bar M,\phi,\xi,\eta,g)$ is called an almost contact metric manifold.

The above relations imply the following useful identities:
\begin{equation}\label{e3}
	g(\phi {X_1},{X_2})=-g({X_1},\phi {X_2}),\qquad
	\eta({X_1})=g({X_1},\xi),
\end{equation}
for all ${X_1},{X_2}\in\Gamma(T\bar M)$.
The covariant derivative of $\phi$ defines as follows
\begin{equation}\label{e11}
	(\bar\nabla_{X_1}\phi){X_2}=\bar\nabla_{X_1}\phi {X_2}-\phi\bar\nabla_{X_1}{X_2},
\end{equation}for any ${X_1},{X_2}\in\Gamma(T\bar M)$.
Almost contact metric structures $(\bar M,\phi,\xi,\eta,g)$ are said to define a Kenmotsu
structure on $\bar M$ if the following characterizing tensorial equation is satisfied (cf. \cite{ken1})
\begin{equation}\label{e4}
	(\bar\nabla_{X_1}\phi){X_2}=g(\phi {X_1},{X_2})\xi-\eta({X_2})\phi {X_1},
\end{equation}
where $\bar\nabla$ denotes the Levi-Civita connection of the ambient manifold. As a consequence of the above relation, we have
\begin{equation}\label{e5}
	\bar\nabla_{X_1}\xi={X_1}-\eta({X_1})\xi.
\end{equation}

The associated fundamental $2$-form $\Phi$ of an almost contact metric manifold is defined by
\begin{equation}\label{ex1}
	\Phi({X_1},{X_2})=g({X_1},\phi {X_2}),
\end{equation}
for all ${X_1},{X_2}\in\Gamma(T\bar M)$.

Motivated by the concept of partially totally real submanifolds introduced by B. Sahin \cite{firat}, we introduce a new class of partially anti-invariant submanifolds, which may be regarded as an odd-dimensional counterpart of this notion in the setting of almost contact metric manifolds.

\begin{definition}
	A submanifold $S$ of an almost contact metric manifold
	$(\bar M,\phi,\xi,\eta,g)$ such that the structure vector field $\xi$
	is tangent to $S$, is said to be a \emph{partially anti-invariant}
	submanifold (PAInv-submanifold) if there exist two
	mutually orthogonal complementary distributions $\mathfrak{D}$ and
	$\mathfrak{D^\perp}$ on $S$ satisfying
	\begin{equation}\label{e80}
		TS=\mathfrak{D}\oplus\mathfrak{D^\perp}\oplus\langle\xi\rangle,
	\end{equation}
	where $\mathfrak{D^\perp}$ is an anti-invariant distribution, i.e.,
	\begin{equation}\label{e90}
		\phi\mathfrak{D^\perp}\subseteq TS^{\perp},
	\end{equation}
	whereas $\mathcal D^{c}$ is called the \emph{ambiguous distribution}
	satisfying
	\begin{equation}\label{e100}
		P\mathfrak{D}\subseteq\mathfrak{D},
		\qquad
		\mathfrak{D}\perp\mathfrak{D^\perp}.
	\end{equation} 
\end{definition}
	
	
	\begin{example}\cite{shk1}
		Every CR-submanifold of a Kenmotsu manifold is a PAInv-submanifold whose ambiguous distribution $\mathfrak{D}$ is invariant.
	\end{example}
	
	\begin{example}\cite{shk2}
		Every pseudo-slant submanifold of a Kenmotsu manifold is a PAInv-submanifold whose ambiguous distribution $\mathfrak{D}$ is a slant submanifold.
	\end{example}

Let $S$ be an immersed submanifold of a Kenmotsu manifold
$(\bar M,\phi,\xi,\eta,g)$. Throughout the paper, we consider the case in which the characteristic vector field $\xi$ is tangent to $S$. For any tangent vector field ${X_1}\in\Gamma(TS)$, we get
\begin{equation}\label{e9}
	\phi {X_1}=T{X_1}+N{X_1},
\end{equation}
where $T{X_1}$ and $N{X_1}$ represent the tangent and normal projections of $\phi {X_1}$, respectively.

Similarly, for every normal vector field $\beta\in\Gamma(TS^\perp)$, the vector field $\phi\beta$ admits the decomposition
\begin{equation}\label{e17}
	\phi\beta=t\beta+f\beta,
\end{equation}
where $t\beta$ and $f\beta$ denote the tangential and normal components of $\phi\beta$, respectively.

For a submanifold $S$ of $\bar M$, the Gauss and Weingarten formulas are given by
\begin{equation}\label{e6}
	\bar\nabla_{X_1}{X_2}=\nabla_{X_1}{X_2}+h({X_1},{X_2}),
\end{equation}
\begin{equation}\label{e7}
	\bar\nabla_{X_1}\beta=-A_\beta {X_1}+\nabla_{X_1}^\perp \beta,
\end{equation}
for all ${X_1},{X_2}\in\Gamma(TS)$ and $\beta\in\Gamma(TS^\perp )$, where $h$ is the second fundamental form, $A_\beta$ is the shape operator, and $\nabla^\perp$ is the normal connection. The two are related with relation 
\begin{equation}\label{e8}
	g(\mathcal{A}_{\beta}{X_1},{X_2})=g(h({X_1},{X_2}),\beta).
\end{equation} 

A Kenmotsu space form $\bar M(c)$ is a Kenmotsu manifold of constant
$\phi$-sectional curvature $c$. In this case, the Riemannian curvature
tensor $\tilde{\mathcal R}$ is given by

\begin{equation}\label{e10}
	\begin{split}
		\tilde{\mathcal{R}}({X_1},{X_2})Z =& \frac{c-3}{4} \{ g({X_2},Z){X_1} - g({X_1},Z){X_2}\}
		+\frac{c+1}{4} \big[\eta({X_1})\eta(Z){X_2} - \eta({X_2})\eta(Z){X_1}\\
		&+g({X_1},Z)\eta({X_2})\xi-g({X_2},Z)\eta({X_1})\xi
		+g(\phi {X_2}, Z)\phi {X_1}\\
		&-g(\phi {X_1}, Z)\phi {X_2} 
		-2g(\phi {X_1}, {X_2})\phi Z\big],
	\end{split}
\end{equation}for all ${X_1}, {X_2}, Z \in \Gamma (TS)$. The covariant derivative of $h$ with respect to the connection in $TS \oplus TS^{\bot}$ is given by
\begin{equation}\label{e12}
	(\bar{\nabla}_{{X_1}}h)({X_2},Z) = {\nabla_{{X_1}}^{\bot}}h({X_2},Z)- h({\nabla}_{{X_1}} {X_2}, Z)-h ({X_2}, {\nabla}_{{X_1}} Z),
\end{equation}and for this, the equation of Gauss and Codazzi's are given as follows:
\begin{equation}\label{e13}	{\mathcal{R}}({X_1},{X_2};Z,W)=\tilde{\mathcal{R}}({X_1},{X_2};Z,W)+g(h({X_1},W),h({X_2},Z))-g(h({X_1},Z),h({X_2},W)),
\end{equation}
\begin{equation}\label{e14}
	(\tilde{\mathcal{R}}({X_1}{X_2})Z)^\perp=(\bar\nabla_{{X_1}}h)({X_2},Z)-(\bar\nabla_{{X_2}}h)({X_1},Z),
\end{equation} for any ${X_1}, {X_2}, Z, W \in \Gamma(TS),$ where ${\mathcal{R}}({X_1},{X_2}; Z,W)= g({\mathcal{R}}({X_1}, {X_2})Z, W).$
For a Kenmotsu manifold $\bar M$, the $\phi$-sectional curvature is given by 
\begin{equation}\label{e15}
	\tilde K_{\phi}({X_1})=\tilde R({X_1},\phi {X_1};\phi {X_1},{X_1}),
\end{equation}for any ${X_1}\in\Gamma(TS)$. In this context, $\phi$-bi-sectional curvature for any ${X_1},{X_2}\in\Gamma(TS)$ computed as:
\begin{equation}\label{e16}
	\tilde K_{\phi {B}}({X_1},{X_2})=\frac{\tilde R({X_1},\phi {X_1};\phi {X_2},{X_2})}{||{X_1}||^2||{X_2}||^2-g({X_1},{X_2})^2}.
\end{equation}

\begin{proposition}\label{P1}
	Let $S$ be a PAInv-submanifold of an almost contact metric manifold $\bar M$. If the ambiguous distribution $D^{c}$ is one-dimensional, then $D^{c}$ is necessarily anti-invariant. Consequently, $S$ is an anti-invariant submanifold of $\bar M$.
\end{proposition}

\begin{proof}
	Suppose that $\dim (D^{c})=1$, and let ${Y_1}$ be a non zero vector field in $D^{c}$, we get $g(\phi {Y_1},{Y_1})=0,$
	we conclude that $\phi {Y_1}$ is orthogonal to ${Y_1}$. As $D^{c}$ is spanned by ${Y_1}$, it follows that
	\begin{equation}
		\phi {Y_1}\notin\Gamma(D^{c}).
	\end{equation}
		Moreover, for every ${X_1}\in\Gamma(D^{\perp})$, we have
	\begin{equation}
		g(\phi {Y_1},{X_1})=-g({Y_1},\phi {X_1})=0,
	\end{equation}
	since $\phi {X_1}$ is normal to $S$. Thus, $\phi {Y_1}$ is orthogonal to $D^{\perp}$ and therefore has no tangential component along $D^{\perp}$.
	Hence, $\phi D^{c}\subset \Gamma(TS^{\perp}).$
	which proves that $S$ is an anti-invariant submanifold of $\bar M$.
\end{proof}

 The normal bundle of a PAInv-submanifold $S$ of a Kenmotsu manifold $\bar M$ admits the orthogonal decomposition
 \begin{equation}
 	TS^\perp =\phi\mathfrak{D^\perp}\oplus N\mathfrak{D}\oplus\mu,
 \end{equation}
 where $\mu$ is a differentiable vector subbundle of $TS^{\perp}$ satisfying
 \begin{equation}
 	\phi \mathfrak{D^\perp}\perp N\mathfrak{D}, \qquad
 	\phi \mathfrak{D^\perp}\perp\mu, \qquad
 	N\mathfrak{D}\perp\mu.
 \end{equation}

\begin{lemma}
	Let $S$ be a PAInv-submanifold of an almost contact metric manifold $\bar M$, and let $\mu$ denote the orthogonal complementary subbundle of $\phi \mathfrak{D^\perp}\oplus N\mathfrak{D}$ in $\Gamma(TS^\perp) $. Then $\mu$ is invariant under $\phi$.
\end{lemma}
\begin{proof}
	Let $\beta\in\Gamma(\mu)$. We shall show that $\phi\beta\in\Gamma(\mu)$. For any ${X_1}\in\Gamma(\mathfrak{D^\perp})$, using \eqref{e2}, we have
	$g(\phi {X_1},\phi\beta)=g({X_1},\beta)-\eta({X_1})\eta(\beta)=0$,
	since $\beta$ is a normal vector field. Hence, $\phi\beta$ is orthogonal to $\phi \mathfrak{D^\perp}$. Also,
	$g(\phi\beta,{X_1})=-g(\beta,\phi {X_1})=0$,
	because $\phi {X_1}\in\phi \mathfrak{D^\perp}$ and $\mu\perp\phi \mathfrak{D^\perp}$. Therefore, $\phi\beta$ is orthogonal to $\mathfrak{D^\perp}$.
	
	Now let ${Y_1}\in\Gamma(\mathfrak{D}\oplus\langle\xi\rangle)$. Then
	$g(\phi\beta,{Y_1})=-g(\beta,\phi {Y_1})=-g(\beta,N{Y_1})=0$,
	since $N\mathfrak{D}$ is orthogonal to $\mu$. Thus, $\phi\beta$ has no component along $\mathfrak{D}\oplus\langle\xi\rangle$. Furthermore, by \eqref{e2},
	$g(\phi\beta,N{Y_1})=g(\beta,{Y_1})-\eta({Y_1})\eta(\beta)=0$,
	which shows that $\phi\beta$ is also orthogonal to $N\mathfrak{D}$.
	
	Consequently, $\phi\beta$ is orthogonal to each of the mutually orthogonal subbundles $\mathfrak{D^\perp}$, $\mathfrak{D}\oplus\langle\xi\rangle$, $\phi \mathfrak{D^\perp}$ and $N\mathfrak{D}$. Hence, the only possible component of $\phi\beta$ belongs to $\mu$. Therefore, $\phi\beta\in\Gamma(\mu)$, proving that $\mu$ is invariant under $\phi$.
\end{proof}
\begin{definition}
	A PAInv-submanifold is said to be a proper if the distribution $\mathfrak{D}$ neither invariant nor anti-invariant, a slant or a pointwise slant distribution.
\end{definition}
\begin{lemma}\label{L3}
	Let $S$ be a PAInv-submanifold of a Kenmotsu manifold $\bar M$. Then, we have
\begin{equation*}
	g(\mathcal{A}_{\mathcal{N}{X_1}}{X_2},W)=g(\mathcal{A}_{\mathcal{N}{X_2}}{X_1},W),
\end{equation*}for any ${X_1},{X_2},W \in\Gamma(\mathfrak{D^\perp})$.
\end{lemma} 
\begin{proof}In the light of \eqref{e7}, we get
	$g(\mathcal{A}_{\phi {X_1}}{X_2},W)=-g(\bar\nabla_{{X_2}}\phi {X_1},W),$
for any ${X_1},{X_2},W \in\Gamma(\mathfrak{D^\perp})$. By virtue of \eqref{e4} and \eqref{e6}, we have
\begin{equation*}
g(\mathcal{A}_{\phi {X_1}}{X_2},W)=g(h({X_2},W),\phi {X_1}).
\end{equation*}This completes the proof.
\end{proof}

\begin{proposition}\label{P4}
	Let $S$ be a PAInv-submanifold of a Kenmotsu manifold $\bar M$. Then anti-invariant distribution is integrable.
\end{proposition}
\begin{proof}
	For any ${Y_1}\in \Gamma(\mathfrak{D}\oplus\langle\xi\rangle)$ and ${X_1},{X_2}\in\Gamma(\mathfrak{D^\perp})$, we may write
	\begin{equation*}
		\begin{split}
		0&=d\Phi({Y_1},{X_2},{X_1})\\
		&=\frac{1}{3}[{Y_1}\Phi({X_2},{X_1})-{X_2}\Phi({Y_1},{X_1})-{X_1}\Phi({Y_1},{X_2})-\Phi([{Y_1},{X_2}],{X_1})\\
		&~~~+\Phi([{Y_1},{X_1}],{X_2})
		-\Phi([{X_2},{X_1}],{Y_1})].
		\end{split}
	\end{equation*}From \eqref{ex1}, we have
	\begin{equation*}
		\begin{split}
		&\frac{1}{3}[{Y_1}g({X_2},{X_1})-{X_2}g({Y_1},\phi {X_1})+{X_1}g({Y_1},\phi {X_2})-g([{Y_1},{X_2}],\phi {X_1})\\
		&~~~+g([{Y_1},{X_1}],\phi {X_2})
		-g([{X_2},{X_1}],\phi {Y_1})]=0.
		\end{split}
	\end{equation*}By the mutual orthogonality of the vector fields, the first five terms on the left-hand side vanish. Consequently, from \eqref{e9}, we have
	\begin{equation*}
	g([{X_2},{X_1}],T{Y_1})=0.
	\end{equation*}
	Hence, $[{X_2},{X_1}]\in\Gamma(\mathfrak{D^\perp})$. This completes the proof.
\end{proof}




\begin{lemma}\label{P5}
Let $S$ be a PAInv-submanifold of a Kenmotsu manifold $\bar M$. Then the distribution $\mathfrak{D}\oplus\langle\xi\rangle$ is integrable if and only if 
\begin{equation*}
	h({Y_1},T{Y_2})-h({Y_2},T{Y_1})+\nabla_{Y_1}^\perp N{Y_2}-\nabla_{Y_2}^\perp N{Y_1} \in \Gamma(N\mathfrak{D}\oplus\mu),
\end{equation*}for any ${Y_1},{Y_2}\in\Gamma(\mathfrak{D}\oplus\langle\xi\rangle)$.
\end{lemma}

\begin{proof}
	By virtue of \eqref{e4}, \eqref{e7}, \eqref{e8} and \eqref{e9}, we have
\begin{equation*}
	g({Y_1},{Y_2})\xi-\eta({Y_2}){Y_1}=\bar\nabla_{Y_1}T{Y_2}+\bar\nabla_{Y_1}N{Y_2}-\phi\nabla_{Y_1}{Y_2}-\phi h({Y_1},{Y_2}),
\end{equation*}for any ${Y_1},{Y_2}\in\Gamma(\mathfrak{D}\oplus\langle\xi\rangle)$. Furthermore, we get
\begin{equation}\label{e21}
	\nabla_{Y_1}T{Y_2}+h({Y_1},T{Y_2})-\mathcal{A}_{N{Y_2}}{Y_1}+\nabla_{Y_1}^\perp N{Y_2}-g(\phi {Y_1},{Y_2})\xi+\eta({Y_2})\phi {Y_1}=\phi \nabla_{Y_1}{Y_2}+\phi h({Y_1},{Y_2}).
\end{equation}By interchanging the roles of ${Y_1}$ and ${Y_2}$ and using \eqref{e21}, we obtain
\begin{equation}\label{e22}
	\begin{split}
	\phi[{Y_1},{Y_2}]&=\nabla_{Y_1}T{Y_2}-\nabla_{Y_2}T{Y_1}+h({Y_1},T{Y_2})-h({Y_2},T{Y_1})-\mathcal{A}_{N{Y_2}}{Y_1}+\mathcal{A}_{N{Y_1}}{Y_2}\\
	&~~~+\eta({Y_2}){Y_1}-\eta({Y_1}){Y_2}+\nabla_{Y_1}^\perp N{Y_2}-\nabla_{Y_2}^\perp N{Y_1}-2g(\phi {Y_1},{Y_2})\xi.
	\end{split}
\end{equation}Since $\phi {X_1}\in\Gamma(TS^\perp)$ for every ${X_1}\in\Gamma(\mathfrak{D^\perp})$, we may writete
\begin{equation*}
	g(\phi[{Y_1},{Y_2}],\phi {X_1})=g(h({Y_1},T{Y_2})-h({Y_2},T{Y_1})+\nabla_{Y_1}^\perp N{Y_2}-\nabla_{Y_2}^\perp N{Y_1},\phi {X_1}).
\end{equation*}By using \eqref{e2} in the left hand side of the above equation, we get
\begin{equation*}
	g([{Y_1},{Y_2}],{X_1})=g(h({Y_1},T{Y_2})-h({Y_2},T{Y_1})+\nabla_{Y_1}^\perp N{Y_2}-\nabla_{Y_2}^\perp N{Y_1},\phi {X_1}).
\end{equation*}This completes the proof.
\end{proof}

\begin{lemma}\label{L6}
	Let $S$ be a PAInv-submanifold of a Kenmotsu manifold $\bar M$. Then leaves of anti-invariant distribution $\mathfrak{D^\perp}$ are totally geodesic 
	if and only if
	\begin{equation*}
		g(\mathcal{A}_{\phi {X_1}}{X_2},{Y_1})=g(\mathcal{A}_{\mathcal{N}{Y_1}}{X_1},{X_2}),
	\end{equation*}for any ${X_1},{X_2}\in \Gamma(\mathfrak{D^\perp})$ and ${Y_1}\in\Gamma(\mathfrak{D}\oplus\langle\xi\rangle)$.
\end{lemma}

\begin{proof}
	For any ${X_1},{X_2} \in \Gamma(\mathfrak{D^\perp})$ and ${Y_1}\in\Gamma(\mathfrak{D}\oplus\langle\xi\rangle)$, we get
	\begin{equation*}
		g(h({X_1},{X_2}),\phi {X_1})=-g(\bar\nabla_{X_2}\phi {Y_1},{X_1})+g(g(\phi {X_2},{Y_1})\xi-\eta({Y_1})\phi {X_2},{X_1}),
	\end{equation*}From \eqref{e6} and \eqref{e4}. In the light of \eqref{e9} and \eqref{e11}, we have
	\begin{equation*}
		g(h({X_1},{X_2}),\phi {X_1})=-g(\bar\nabla_{X_2}T{Y_1},{X_1})-g(\bar\nabla_{X_2}N{Y_1},{X_1}).
	\end{equation*}By using \eqref{e6} and \eqref{e7}, we get
\begin{equation*}
g(\mathcal{A}_{\phi {X_1}}{X_2},{Y_1})=g(\nabla_{X_2}{Y_1},T{Y_1})+g(\mathcal{A}_{N{Y_1}}{X_1},{X_2}),
\end{equation*}which yields the desired result.
\end{proof}

\begin{lemma}\label{L7}
	Let $S$ be a PAInv-submanifold of a Kenmotsu manifold $\bar M$. Then leaves of distribution $\mathfrak{D}\oplus\langle\xi\rangle$ defines totally geodesic foliations on $S$ if and only if
\begin{equation*}
	g(\mathcal{A}_{\phi {X_1}}{Y_1},{Y_2})=g(\mathcal{A}_{N{Y_2}}{X_1},{Y_1}),
\end{equation*}	for any ${Y_1},{Y_2}\in\Gamma(\mathfrak{D}\oplus\langle\xi\rangle)$ and ${X_1}\in\Gamma(\mathfrak{D^\perp}).$
\end{lemma}

\begin{proof}for any ${Y_1},{Y_2}\in\Gamma(\mathfrak{D}\oplus\langle\xi\rangle)$ and ${X_1}\in\Gamma(\mathfrak{D^\perp}),$ we have
	\begin{equation*}
		g(\mathcal{A}_{\phi {X_1}}{Y_1},{Y_2})=-g(\phi\bar\nabla_{{Y_1}}{X_1},{Y_2})+g((\bar\nabla_{{Y_1}}\phi){X_1},{Y_2}),
	\end{equation*}by using \eqref{e7} and \eqref{e11}. By virtue of \eqref{e4}, we have
	\begin{equation*}
		g(\mathcal{A}_{\phi {X_1}}{Y_1},{Y_2})=-g(\bar\nabla_{{Y_1}}\phi {Y_2},{X_1})+g(\phi {Y_1},{X_1})\eta({Y_2})-\eta({X_1})g(\phi {Y_1},{Y_2}).
	\end{equation*}By using the orthogonality of vector fields, \eqref{e9}, \eqref{e6} and \eqref{e7}, we can write
	\begin{equation*}
		g(\mathcal{A}_{\phi {X_1}}{Y_1},{Y_2})=-g(\nabla_{{Y_1}}T{Y_2},{X_1})+g(\mathcal{A}_{N{Y_2}}{Y_1},{X_1}).
	\end{equation*}Hence, the desired result follows immediately from \eqref{e8}.
\end{proof}




\section{PAInv-Submanifolds Admitting a Parallel Canonical Structure}

We now investigate the consequences of the parallelism of the canonical structure tensors \(T\) and \(N\) on the geometry of PAInv-submanifolds of a Kenmotsu manifold. In particular, we examine how the conditions \(\bar{\nabla}T=0\) and \(\bar{\nabla}N=0\) affect the geometric properties of PAInv-submanifolds. The covariant derivative of the tensor field \(T\) is defined by
\begin{equation}\label{e30}
	(\bar{\nabla}_{X_1} T){X_2}=\nabla_{X_1}T{X_2}-T\nabla_{X_1}{X_2},
\end{equation}
for any \({X_1},{X_2}\in\Gamma(TS)\). From \eqref{e11}, \eqref{e4}, \eqref{e7}, \eqref{e8} and \eqref{e9}, we get
\begin{equation}\label{e31}
	\begin{split}
	\nabla_{X_1}T{X_2}+h({X_1},T{X_2})-\mathcal{A}_{N{X_2}}{X_1}+\nabla_{X_1}^\perp N{X_2}-g(\phi {X_1},{X_2})\xi+\eta({X_2})\phi {X_1}\\=T\nabla_{X_1}{X_2}+N\nabla_{X_1}{X_2}+\phi h({X_1},{X_2}).
	\end{split}
\end{equation} By virtue of \eqref{e1}, \eqref{e2} and \eqref{e17}, we get
\begin{equation}\label{e32}
	(\bar\nabla_{{X_1}}T){X_2}=\mathcal{A}_{N{X_2}}{X_1}+th({X_1},{X_2})+g(T{X_1},{X_2})\xi-\eta({X_2})T{X_1}.
\end{equation}Taking inner product with $Z\in\Gamma(TS)$, we have
\begin{equation}\label{e33}
g((\bar\nabla_{{X_1}}T){X_2},Z)=g(\mathcal{A}_{N{X_2}}Z-\mathcal{A}_{NZ}{X_2},{X_1})+g(T{X_1},{X_2})\eta(Z)-\eta({X_2})g(T{X_1},Z).
\end{equation} By applying \eqref{e31} and setting $(\bar\nabla_{X_1}N){X_2}=\nabla_{X_1}^\perp N{X_2}-N\nabla_{X_1}{X_2}$, we may write
\begin{equation}\label{e34}
	(\bar\nabla_{X_1}N){X_2}=fh({X_1},{X_2})-h({X_1},T{X_2})-\eta({X_2})N{X_1}.
\end{equation}From \eqref{e32}, we may conclude the following Lemma.

\begin{lemma}\label{L9}
	Let $S$ be a PAInv-submanifold of a Kenmotsu manifold $\bar M$. Then $T$ is parallel if and only 
	\begin{equation*}
		\mathcal{A}_{N{X_2}}{X_1}=\mathcal{A}_{N{X_1}}{X_2}+g(T{X_1},{X_2})\xi-\eta({X_2})T{X_1},
	\end{equation*}for any ${X_1},{X_2}\in\Gamma(TS)$.
\end{lemma}



\begin{lemma}\label{L11}
		Let $S$ be a PAInv-submanifold of a Kenmotsu manifold $\bar M$. Then $N$ is parallel if and only if 
		\begin{equation*}
			\mathcal{A}_{f\beta}{Y_2}=-\mathcal{A}_{\beta}T{Y_2}+\eta({Y_2})t\beta,
		\end{equation*}for any $\beta\in\Gamma(TS^\perp)$ and ${Y_2}\in\Gamma(TS)$.
\end{lemma}

\begin{proof}
From \eqref{e34}, we obtain $g(h({Y_1},f{Y_2}),\beta)=g(h({Y_1},T{Y_2}),\beta)+\eta({Y_2})g(N{Y_1},\beta),$
for any \(\beta\in\Gamma(TS^\perp)\) and \({Y_1},{Y_2}\in\Gamma(TS)\).
By virtue of \eqref{e17}, we further obtain $-g(\mathcal{A}_{f\beta}{Y_2},{Y_1})=g(\mathcal{A}_{\beta}T{Y_2},{Y_1})-\eta({Y_2})g(t\beta,{Y_1}).$
Hence, the desired result follows immediately.
	\end{proof}
	
	\begin{proposition}\label{P12}
		Let $S$ be a PAInv-submanifold of a Kenmotsu manifold $\bar M$. If $N$ is parallel, then the following statement holds.
		\begin{itemize}
			\item [(i)] $\mathcal{A}_{f\beta}{X_1}=0$, for any $\beta\in\Gamma(TS^\perp)$ and ${X_1}\in\Gamma(\mathfrak{D^\perp})$.
			\item [(ii)] $\mathcal{A}_{\phi {X_1}}{Y_1}\in\Gamma(\mathfrak{D^\perp})$ for any ${Y_1}\in\Gamma(TS)$ and ${X_1}\in\Gamma(\mathfrak{D^\perp})$.
			\item [(iii)] $\phi \mathfrak{D^\perp} \oplus N\mathfrak{D}$ is parallel in the normal bundle.
			\item [(iv)] The ambiguous distribution $\mathfrak{D}\oplus\langle\xi\rangle$ is integrable and leaves are totally geodesic in $S$ if and only if $\nabla_{\mathfrak{D}}^\perp(N\mathfrak{D}\oplus\mu)$ has no component in $\Gamma(\mathfrak{D^\perp})$,
		\end{itemize}
	\end{proposition}

	\begin{proof}
	In the light of \ref{L11}, we get $\mathcal{A}_{f\beta}{X_1}=-\mathcal{A}_{\beta}T{X_1}$, for any $\beta\in\Gamma(TS^\perp)$ and ${X_1}\in\Gamma(\mathfrak{D^\perp})$.
Since, distribution \(\mathfrak{D^\perp}\) is anti-invariant, then \(T{X_1}=0\) for every \({X_1}\in\Gamma(\mathfrak{D^\perp})\). Hence,
$\mathcal{A}_{f\beta}{X_1}=0,$ for all \(\beta\in\Gamma(TS^\perp)\). For the second part, using the assumption on $N$ together with \eqref{e11}, we obtain
	\begin{equation*}
		g(\nabla_{{Y_1}}^\perp N{Y_2},\phi {X_1})=g(\bar\nabla_{{Y_1}}\phi {Y_2},\phi {X_1})-g((\bar\nabla_{Y_1}\phi){Y_2},\phi {X_1}),
	\end{equation*}
	for any ${X_1}\in\Gamma(\mathfrak{D^\perp})$, ${Y_1}\in\Gamma(TS)$, and ${Y_2}\in\Gamma(\mathfrak{D}\oplus\langle\xi\rangle)$. By virtue of \eqref{e4}, we have
	\begin{equation*}
		g(\nabla_{{Y_1}}^\perp N{Y_2},\phi {X_1})=g(\bar\nabla_{{Y_1}}\phi {Y_2},\phi {X_1})-g(g(\phi {Y_1},{Y_2})\xi-\eta({Y_2})\phi {Y_1},\phi {X_1}).
	\end{equation*}
	Furthermore, by using \eqref{e6}, \eqref{e3} and \eqref{e7}, we get
	\begin{equation*}
		g(\nabla_{{Y_1}}^\perp N{Y_2},\phi {X_1})=g(\mathcal{A}_{\phi {X_1}}{Y_1},T{Y_2})-g(\nabla_{Y_1}^\perp \phi {X_1},N{Y_2}).
	\end{equation*}
By using the parallelism of $T$, we may write
$g(\mathcal{A}_{\phi {X_1}}{Y_1},T{Y_2})=0.$
Hence, the second part of the result follows. By using \eqref{e1}, \eqref{e6}, and \eqref{e7}, we have
	\begin{equation*}
		g(\phi {Y_1},{X_1})\xi-\eta({X_1})\phi {Y_1}=-\mathcal{A}_{\phi {X_1}}{Y_1}+\nabla_{{Y_1}}^\perp \phi {X_1}-\phi \nabla_{Y_1}{X_1}-\phi h({Y_1},{X_1}).
	\end{equation*}
	The inner product of the above equation with $\beta\in\Gamma(\mu)$, combined with \eqref{e8}, and \eqref{e17}, gives
	\begin{equation*}
		g(\nabla_{{Y_1}}^\perp \phi {X_1},\beta)=g(\phi \nabla_{Y_1}{X_1},\beta)+g((\bar\nabla_{Y_1}\phi){X_1},\beta)+g(\phi h({Y_1},{X_1}),\beta).
	\end{equation*}
	Consequently, from \eqref{e4}, we may write
	\begin{equation}\label{e35}
		g(\nabla_{{Y_1}}^\perp \phi {X_1},\beta)=g(\mathcal{A}_{f\beta}{X_1},{Y_1}).
	\end{equation}
	
By using the parallelism of \(N\), for any \({Y_2}\in\Gamma(\mathfrak{D}\oplus\langle\xi\rangle)\), we have
	\begin{equation}\label{e36}
		g(\nabla_{{Y_1}}^\perp N{Y_2},\beta)=-g(\nabla_{Y_1}{Y_2},N\beta).
	\end{equation}
	
	From \eqref{e35} in conjunction with \eqref{e36}, we obtain
	\begin{equation*}
		g(\nabla_{{Y_1}}^\perp \phi {X_1}+\nabla_{{Y_1}}^\perp N{Y_2},\beta)
		=g(\mathcal{A}_{f\beta}{X_1},{Y_1}).
	\end{equation*}
In the light of part \((i)\) of the proposition, which states that \(\mathcal{A}_{f\beta}{X_1}=0\), we obtain
	\begin{equation*}
		g(\nabla_{{Y_1}}^\perp \phi {X_1}+\nabla_{{Y_1}}^\perp N{Y_2},\beta)=0,
	\end{equation*}
This completes part (iii) of the proposition. For the proof of part (iv), we have
	\begin{equation*}
		g(\nabla_{Y_2}{X_1},Z)=g(\nabla_{Y_2}^\perp \phi {X_1},NZ)-g(\mathcal{A}_{\phi {X_1}}{Y_2},TZ)+g(g(\phi {Y_2},{X_1})\xi-\eta({X_1})\phi {Y_2},\phi Z).
	\end{equation*}for ${Y_2},Z\in\Gamma(\mathfrak{D}\oplus\langle\xi\rangle)$ and ${X_1}\in\Gamma(\mathfrak{D^\perp})$, by using \eqref{e1}, \eqref{e4}, \eqref{e7} and \eqref{e9}.
	By part $(ii)$ of the proposition, which asserts that
	$\mathcal{A}_{\phi {X_1}}{Y_1}\in\Gamma(\mathfrak{D^\perp})$ for every
	${Y_1}\in\Gamma(TS)$, we obtain
	\begin{equation}\label{e37}
		g(\nabla_{Y_2}Z,{X_1})=g(\nabla_{Y_2}^\perp NZ,\phi {X_1}).
	\end{equation}
	
	Furthermore, by part $(i)$ of the proposition, namely,
	$\mathcal{A}_{f\beta}{X_1}=0$ for every $\beta\in\Gamma(\mu)$, equation
	\eqref{e35} simplifies to
	\begin{equation}\label{e38}
		g(\nabla_{Y_2}^\perp\phi {X_1},\beta)=0.
	\end{equation}
	
	It follows from \eqref{e37} and \eqref{e38} that
	\begin{equation*}
		g(\nabla_{Y_2}Z,{X_1})=g(\nabla_{Y_2}^\perp(NZ+\beta),\phi {X_1}).
	\end{equation*}
	Therefore, \(\mathfrak{D}\oplus\langle\xi\rangle\) is totally geodesic if and only if
	\begin{equation*}
		g(\nabla_{Y_2}^\perp(NZ+\beta),\phi {X_1})=0,
	\end{equation*}
	which completes the proof.
	\end{proof}

	\section{Characterization of PAInv-Submanifolds of Kenmotsu Space Forms}
	
The following results provide characterizations of the curvature, sectional curvature, and bi-sectional curvature of Kenmotsu space forms. Applying \eqref{e10}, for any
${Y_1},{Y_2}\in\Gamma(\mathfrak{D}\oplus\langle\xi\rangle)$ and
${X_1},{X_2}\in\Gamma(\mathfrak{D^\perp})$, we obtain
	\begin{equation}\label{e39}
		g(\tilde{R}({Y_1},T{Y_1}){X_1},\phi {X_1})=-\frac{(c+1)}{2}||T{Y_1}||^2||\phi {X_1}||^2.
	\end{equation} By virtue of \eqref{e10} and \eqref{e8}, we have
	\begin{equation}\label{e40}
		\begin{split}
		g(\tilde{R}({Y_1},{Y_2}){X_1},\phi {X_2})&=g((\bar\nabla_{{Y_1}}h)({Y_2},{X_1})-(\bar\nabla_{{Y_2}}h)({Y_1},{X_1}),\phi {X_2})\\
		&=g(\nabla_{{Y_1}}^\perp h({Y_2},{X_1})-h(\nabla_{{Y_1}}{Y_2},{X_1})-h({Y_2},\nabla_{{Y_1}}{X_1}),\phi {X_2})\\
		&-g(\nabla_{{Y_2}}^\perp h({Y_1},{X_1})-h(\nabla_{{Y_2}}{Y_1},{X_1})-h({Y_1},\nabla_{{Y_2}}{X_1}),\phi {X_2}).
		\end{split}
	\end{equation} By applying \eqref{e8} to the third and sixth terms, we obtain
	
	\begin{equation}\label{e41}
		\begin{split}
			g(\tilde{R}({Y_1},{Y_2}){X_1},\phi {X_2})&=g(\nabla_{{Y_1}}^\perp h({Y_2},{X_1}),\phi {X_2})-g(h(\nabla_{{Y_1}}{Y_2},{X_1}),\phi {X_2})\\
			&~~~-g(\nabla_{{Y_2}}^\perp h({Y_1},{X_1}),\phi {X_2})
			+g(h(\nabla_{{Y_2}}{Y_1},{X_1}),\phi {X_2})\\
			&~~~-g(\mathcal{A}_{\phi {X_2}}{Y_2},\nabla_{{Y_1}}{X_1})+g(\mathcal{A}_{\phi {X_2}}{Y_1},\nabla_{{Y_2}}{X_1}).
		\end{split}
	\end{equation}Operating $\phi$ to the last two terms of the above equation and using \eqref{e2}, we obtain
	
		\begin{equation}\label{e41}
		\begin{split}
			g(\tilde{R}({Y_1},{Y_2}){X_1},\phi {X_2})&=g(\nabla_{{Y_1}}^\perp h({Y_2},{X_1}),\phi {X_2})-g(h(\nabla_{{Y_1}}{Y_2},{X_1}),\phi {X_2})\\
			&~~~-g(\nabla_{{Y_2}}^\perp h({Y_1},{X_1}),\phi {X_2})+g(h(\nabla_{{Y_2}}{Y_1},{X_1}),\phi {X_2})\\
			&~~~-g(\phi\mathcal{A}_{\phi {X_2}}{Y_2},\phi\nabla_{{Y_1}}{X_1})+g(\phi\mathcal{A}_{\phi {X_2}}{Y_1},\phi\nabla_{{Y_2}}{X_1})\\
			&~~~-\eta(\mathcal{A}_{\phi {X_2}}{Y_2})\eta(\nabla_{{Y_1}}{X_1})+\eta(\mathcal{A}_{\phi {X_2}}{Y_1})\eta(\nabla_{{Y_2}}{X_1}).
		\end{split}
	\end{equation}The second-last term in the above equation vanishes since
	$\eta(\nabla_{{Y_1}}{X_1})=-g(\nabla_{{Y_1}}\xi,{X_1}),$
	where, by \eqref{e5}, $\nabla_{{Y_1}}\xi=-\phi {Y_1}$. As $\phi {Y_1}$ is orthogonal to ${X_1}$, we conclude that
	$\eta(\nabla_{{Y_1}}{X_1})=0$. Similarly, $\eta(\nabla_{{Y_2}}{X_1})=0$. Furthermore, using \eqref{e11}, we obtain
	
		\begin{equation*}
		\begin{split}
			g(\tilde{R}({Y_1},{Y_2}){X_1},\phi {X_2})&=g(\nabla_{{Y_1}}^\perp h({Y_2},{X_1}),\phi {X_2})-g(\nabla_{{Y_2}}^\perp h({Y_1},{X_1}),\phi {X_2})\\
			&~~~-g(h(\nabla_{{Y_1}}{Y_2},{X_1}),\phi {X_2})+g(h(\nabla_{{Y_2}}{Y_1},{X_1}),\phi {X_2})\\
			&~~~-g(\phi\mathcal{A}_{\phi {X_2}}{Y_2},\bar\nabla_{{Y_1}}\phi {X_1})+g(\phi\mathcal{A}_{\phi {X_2}}{Y_2},(\bar\nabla_{{Y_1}}\phi){X_1})\\
			&~~~+g(\phi\mathcal{A}_{\phi {X_2}}{Y_1},\bar\nabla_{{Y_2}}\phi {X_1})-g(\phi\mathcal{A}_{\phi {X_2}}{Y_1},(\bar\nabla_{{Y_2}}\phi){X_1}).
		\end{split}
	\end{equation*}In view of \eqref{e4}, the sixth and eighth terms vanish. Thus, we arrive at
		\begin{equation}\label{e42}
		\begin{split}
			g(\tilde{R}({Y_1},{Y_2}){X_1},\phi {X_2})&=g(\nabla_{{Y_1}}^\perp h({Y_2},{X_1}),\phi {X_2})-g(\nabla_{{Y_2}}^\perp h({Y_1},{X_1}),\phi {X_2})\\
			&~~~-g(h(\nabla_{{Y_1}}{Y_2},{X_1}),\phi {X_2})+g(h(\nabla_{{Y_2}}{Y_1},{X_1}),\phi {X_2})\\
			&~~~+g(\mathcal{A}_{\phi {X_2}}{Y_2},\phi\bar\nabla_{{Y_1}}\phi {X_1})-g(\mathcal{A}_{\phi {X_2}}{Y_1},\phi\bar\nabla_{{Y_2}}\phi {X_1}).
		\end{split}
	\end{equation} In the light of \eqref{e7}, \eqref{e8} and \eqref{e9}, we have
	
	\begin{equation}\label{e43}
		\begin{split}
			g(\tilde{R}({Y_1},{Y_2}){X_1},\phi {X_2})&=g(\nabla_{{Y_1}}^\perp h({Y_2},{X_1}),\phi {X_2})-g(\nabla_{{Y_2}}^\perp h({Y_1},{X_1}),\phi {X_2})\\
			&~~~-g(\mathcal{A}_{\phi {X_2}}{Y_2},T\mathcal{A}_{\phi {X_1}}{Y_1})-g(\nabla_{{Y_1}}^\perp \phi {X_1},N\mathcal{A}_{\phi {X_2}}{Y_2})
			\\
			&~~~+g(\mathcal{A}_{\phi {X_2}}{Y_1},T\mathcal{A}_{\phi {X_1}}{Y_2})+g(\nabla_{{Y_2}}^\perp \phi {X_1},N\mathcal{A}_{\phi {X_2}}{Y_1})\\
			&~~~-g(h([{Y_1},{Y_2}],{X_1}),\phi {X_2}).
		\end{split}
	\end{equation}Specializing to the case ${Y_2}=T{Y_1}$ and ${X_2}={X_1}$, we obtain
	
	\begin{equation*}
		\begin{split}
			g(\tilde{R}({Y_1},T{Y_1}){X_1},\phi {X_1})&=g(\nabla_{{Y_1}}^\perp h(T{Y_1},{X_1}),\phi {X_1})-g(\nabla_{T{Y_1}}^\perp h({Y_1},{X_1}),\phi {X_1})\\
			&~~~-g(\mathcal{A}_{\phi {X_1}}T{Y_1},T\mathcal{A}_{\phi {X_1}}{Y_1})-g(\nabla_{{Y_1}}^\perp \phi {X_1},N\mathcal{A}_{\phi {X_1}}T{Y_1})
			\\
			&~~~+g(\mathcal{A}_{\phi Z}{Y_1},T\mathcal{A}_{\phi {X_1}}T{Y_1})+g(\nabla_{T{Y_1}}^\perp \phi {X_1},N\mathcal{A}_{\phi {X_1}}{Y_1})\\
			&~~~-g(h([{Y_1},T{Y_1}],{X_1}),\phi {X_1}).
		\end{split}
	\end{equation*} Simplifying the above expression, we obtain
		\begin{equation}\label{e44}
		\begin{split}
			g(\tilde{R}({Y_1},T{Y_1}){X_1},\phi {X_1})&=g(\nabla_{{Y_1}}^\perp h(T{Y_1},{X_1}),\phi {X_1})-g(\nabla_{T{Y_1}}^\perp h({Y_1},{X_1}),\phi {X_1})\\
			&~~~-2g(\mathcal{A}_{\phi {X_1}}T{Y_1},T\mathcal{A}_{\phi {X_1}}{Y_1})-g(\nabla_{{Y_1}}^\perp \phi {X_1},N\mathcal{A}_{\phi {X_1}}T{Y_1})
			\\
			&~~~+g(\nabla_{T{Y_1}}^\perp \phi {X_1},N\mathcal{A}_{\phi {X_1}}{Y_1}-g(h([{Y_1},T{Y_1}],{X_1}),\phi {X_1}).
		\end{split}
	\end{equation}By using \eqref{e39} together with \eqref{e44}, we obtain
	
		\begin{equation}\label{ecurv}
		\begin{split}
			-\frac{(c+1)}{2}||T{Y_1}||^2||\phi {X_1}||^2&=g(\nabla_{{Y_1}}^\perp h(T{Y_1},{X_1}),\phi {X_1})-g(\nabla_{T{Y_1}}^\perp h({Y_1},{X_1}),\phi {X_1})\\
			&~~~-2g(\mathcal{A}_{\phi {X_1}}T{Y_1},T\mathcal{A}_{\phi {X_1}}{Y_1})-g(\nabla_{{Y_1}}^\perp \phi {X_1},N\mathcal{A}_{\phi {X_1}}T{Y_1})
			\\
			&~~~+g(\nabla_{T{Y_1}}^\perp \phi {X_1},N\mathcal{A}_{\phi {X_1}}{Y_1}-g(h([{Y_1},T{Y_1}],{X_1}),\phi {X_1}).
		\end{split}
	\end{equation}
	
	\begin{lemma}\label{L13}
	Let $S$ be a PAInv-submanifold of a Kenmotsu space form $\bar{M}(c)$. Assume that $\mathfrak{D}\oplus\langle\xi\rangle$ is integrable and that $\nabla_{\mathfrak{D}}^{\perp}(N\mathfrak{D}\oplus\mu)$ has no component in $\Gamma(\phi \mathfrak{D^\perp})$. If $S$ is mixed totally geodesic or
	$g\bigl(h(\mathfrak{D},\mathfrak{D^\perp}),\phi \mathfrak{D^\perp}\bigr)=0,$
	then
		 \begin{equation}\label{e45}
		 	\frac{(c+1)}{2}||T{Y_1}||^2||\phi {X_1}||^2+2||\mathcal{A}_{\phi {X_1}}T{Y_1}||^2=0,
		 \end{equation}for any ${Y_1}\in\Gamma(\mathfrak{D}\oplus\langle\xi\rangle)$ and ${X_1}\in\Gamma(\mathfrak{D^\perp})$.
	\end{lemma}
	
	\begin{proof}Suppose that $S$ is a PAInv-submanifold of a Kenmotsu space form $\bar{M}(c)$. Then
	\begin{equation*}
	\begin{split}
		g(\tilde{R}({Y_1},T{Y_1}){X_1},\phi {X_1})&=g(\nabla_{{Y_1}}^\perp h(T{Y_1},{X_1}),\phi {X_1})-g((\nabla_{T{Y_1}}^\perp h({Y_1},{X_1}),\phi {X_1})\\
		&~~~-2g(\mathcal{A}_{\phi {X_1}}T{Y_1},T\mathcal{A}_{\phi {X_1}}{Y_1})-g(\nabla_{{Y_1}}^\perp \phi {X_1},N\mathcal{A}_{\phi {X_1}}T{Y_1})
		\\
		&~~~+g(\nabla_{T{Y_1}}^\perp \phi {X_1},N\mathcal{A}_{\phi {X_1}}{Y_1}-g(h([{Y_1},T{Y_1}],{X_1}),\phi {X_1}).
	\end{split}
\end{equation*}
		Furthermore, the first two terms vanish by the assumption that $\nabla_{\mathfrak{D}}^{\perp}(N\mathfrak{D}\oplus\mu)$ has no component in $\Gamma(\phi \mathfrak{D^\perp})$, whereas the last term vanishes as a consequence of the integrability of the distribution $\mathfrak{D}\oplus\langle\xi\rangle$. Hence, the above equation reduces to
		\begin{equation}\label{e46}
			\begin{split}
				g(\tilde{R}({Y_1},T{Y_1}){X_1},\phi {X_1})&=-2g(\mathcal{A}_{\phi {X_1}}T{Y_1},T\mathcal{A}_{\phi {X_1}}{Y_1})-g(\nabla_{{Y_1}}^\perp \phi {X_1},N\mathcal{A}_{\phi {X_1}}T{Y_1})
				\\
				&~~~+g(\nabla_{T{Y_1}}^\perp \phi {X_1},N\mathcal{A}_{\phi {X_1}}{Y_1}).
			\end{split}
		\end{equation}From which, we get
		
		\begin{equation}\label{e47}
			\begin{split}
				g(\tilde{R}({Y_1},T{Y_1}){X_1},\phi {X_1})&=-2g(\mathcal{A}_{\phi {X_1}}T{Y_1},T\mathcal{A}_{\phi {X_1}}{Y_1})+g(\nabla_{{Y_1}}^\perp(N\mathcal{A}_{\phi {X_1}}T{Y_1}), \phi {X_1})
				\\
				&~~~-g(\nabla_{T{Y_1}}^\perp(N\mathcal{A}_{\phi {X_1}}{Y_1}) ,\phi {X_1}).
			\end{split}
		\end{equation}
	By virtue of hypothesis $g(h(\mathfrak{D},\mathfrak{D^\perp}),\phi \mathfrak{D^\perp})=0$ provides $\mathcal{A}_{\phi {X_1}}T{Y_1}\in\Gamma(\mathfrak{D}\oplus\langle\xi\rangle)$ and hence $N\mathcal{A}_{\phi {X_1}}T{Y_1}\in\Gamma(N\mathfrak{D})$. Therefore, together with the assumption that $\nabla_{\mathfrak{D}}^{\perp}(N\mathfrak{D}\oplus\mu)$ has no component along $\phi \mathfrak{D^\perp}$, the second and third terms in the right hand side vanishes. Then, we have
	\begin{equation}\label{e48}
			g(\tilde{R}({Y_1},T{Y_1}){X_1},\phi {X_1})=-2g(\mathcal{A}_{\phi {X_1}}T{Y_1},T\mathcal{A}_{\phi {X_1}}{Y_1}).
	\end{equation}By using Lemma \ref{P5}, followed by taking the inner product with
	$\phi {X_1}$ for any ${X_1}\in\Gamma(\mathfrak{D^\perp})$, and making use of \eqref{e8}, we arrive at
	\begin{equation}
		g(\mathcal{A}_{\phi {X_1}}T{Y_1},{Y_2})=-g(T\mathcal{A}_{\phi {X_1}}{Y_1},{Y_2}),
	\end{equation}from which we have $\mathcal{A}_{\phi {X_1}}T{Y_1}=-T\mathcal{A}_{\phi {X_1}}{Y_1}$. By substituting this into \eqref{e48}, we obtain
	\begin{equation}\label{e51}
	g(\tilde{R}({Y_1},T{Y_1}){X_1},\phi {X_1})=2||\mathcal{A}_{\phi {X_1}}T{Y_1}||^2.
	\end{equation}
	By using \eqref{e39} and \eqref{e51}, we get
	\begin{equation}\label{e52}
	\frac{(c+1)}{2}||T{Y_1}||^2||\phi {X_1}||^2+2||\mathcal{A}_{\phi {X_1}}T{Y_1}||^2=0.
	\end{equation}
	Which completes the proof.
	\end{proof}
	
An immediate consequence of Lemma \ref{L13} is the following result.
	
	\begin{corollary}\label{C14}
		Let $S$ be a mixed totally geodesic PAInv-submanifold in Kenmotsu space form
		$\bar{M}(c)$. If $\mathfrak{D}\oplus\langle\xi\rangle$ is integrable and leaves are totally geodesic in $S$, then $S$ is totally real submanifold.
	\end{corollary}
	
	\begin{proof}
		By using Lemma \ref{L7}, we have
		\begin{equation}\label{e53}
			g(\mathcal{A}_{\phi {X_1}}{Y_1},{Y_2})=g(\mathcal{A}_{N{Y_2}}{X_1},{Y_1}),
		\end{equation}
		for any ${Y_1},{Y_2}\in\Gamma(\mathfrak{D}\oplus\langle\xi\rangle)$ and ${X_1}\in\Gamma(\mathfrak{D^\perp})$.
		\end{proof}
		From \eqref{e8} in \eqref{e53}, we obtain
		\begin{equation*}
		g(h({Y_1},{Y_2}),\phi {X_1})=g(h({Y_1},{X_1}),N{Y_2}).
		\end{equation*}Since, $\mathfrak{D}$ defines totally geodesic foliation $S$, from which, we get $g(\mathcal{A}_{\phi {X_1}}{Y_1},{Y_2})=0$, which implies that $\mathcal{A}_{\phi {X_1}}{Y_1}\in\Gamma(\mathfrak{D^\perp})$.
		
	Furthermore, since $S$ is mixed totally geodesic, it follows that
	$\mathcal{A}_{\phi {X_1}}{Y_1}\in\Gamma(\mathfrak{D}\oplus\langle\xi\rangle)$, which implies that
	$\mathcal{A}_{\phi {X_1}}{Y_1}=0$. By substituting this into \eqref{e45}, we obtain
	$T{Y_1}=0.$	Therefore, $\phi {Y_1}$ has no tangential component. Hence, $S$ is a totally real submanifold.

\begin{corollary}\label{C15}
		Let $S$ be a mixed totally geodesic PAInv-submanifold of a Kenmotsu space form $\bar{M}(c)$. Assume that the distribution $\mathfrak{D}\oplus\langle\xi\rangle$ is integrable with totally geodesic leaves. If the tensor field $T$ is parallel, then $S$ is a totally real submanifold
\end{corollary}.

\begin{proof}
As the distribution $\mathfrak{D}\oplus\langle\xi\rangle$ is integrable and its leaves are totally geodesic in $S$, we obtain \eqref{e52}. On the other hand, equation \eqref{e32} yields
\begin{equation*}
	(\nabla_{{Y_1}}T){X_1}=\mathcal{A}_{N{X_1}}{Y_1}+th({Y_1},{X_1}),
\end{equation*}
for any ${Y_1}\in\Gamma(\mathfrak{D}\oplus\langle\xi\rangle)$ and ${X_1}\in\Gamma(\mathfrak{D^\perp})$.

Since $T$ is parallel, taking the inner product of the above equation with an arbitrary vector field $Z\in\Gamma(TS)$, we obtain
\begin{equation*}
g(\mathcal{A}_{N{X_1}}{Y_1},Z)=g(h({Y_1},{X_1}),NZ).
\end{equation*}
By virtue of the totally geodesicness of $S$ implies that $\mathcal{A}_{N{X_1}}{Y_1}=0$. Consequently, the desired result follows immediately from \eqref{e52}.
\end{proof}
		
\begin{theorem}\label{T16}
		Let $S$ be a proper mixed totally geodesic PAInv-submanifold in Kenmotsu space form
	$\bar{M}(c)$. Assume that $\mathfrak{D}\oplus\langle\xi\rangle$ is integrable and leaves are totally geodesic 
	in $S$. If $N$ is parallel, then $c=-1$. 
\end{theorem}		
		
\begin{proof}
By Proposition \ref{P12} and the parallelism of $N$, we have
$\mathcal{A}_{\phi {X_1}}{Y_1}\in\Gamma(\mathfrak{D^\perp}),$
for every ${X_1}\in\Gamma(\mathfrak{D^\perp})$ and ${Y_1}\in\Gamma(TS)$. Hence, $g(\mathcal{A}_{\phi {X_1}}{Y_1},{Y_2})=0,$
for any $0\neq {Y_2}\in\Gamma(\mathfrak{D}\oplus\langle\xi\rangle)$, which implies that $\mathcal{A}_{\phi {X_1}}{Y_2}=0$. Taking the inner product with ${X_2}\in\Gamma(\mathfrak{D^\perp})$ and using \eqref{e8}, we obtain
\begin{equation*}
	g\bigl(h({Y_2},{X_2}),\phi {X_1}\bigr)=0.
\end{equation*}
Therefore, $g\bigl(h(\mathfrak{D},\mathfrak{D^\perp}),\phi \mathfrak{D^\perp}\bigr)=0.$
On the other hand, by statement $(iv)$ of Proposition \ref{P12}, the distribution $\mathfrak{D}$ is integrable and its leaves are totally geodesic in $S$ if and only if $\nabla_{\mathfrak{D}}^{\perp}(N\mathfrak{D}\oplus\mu)$ has no component in $\Gamma(\phi \mathfrak{D^\perp})$. Hence, \eqref{e52} follows. Combining the above facts, we obtain
\begin{equation*}
	\frac{(c+1)}{2}||T{Y_2}||^2||\phi {X_1}||^2=0,
\end{equation*}
and, since $S$ is a proper PAInv-submanifold, the desired result follows.
\end{proof}
		
\section{Totally Umbilical PAInv-Submanifolds}		
		
We next turn our attention to totally umbilical PAInv-submanifolds of Kenmotsu manifolds. The characterization of such submanifolds is investigated, and the geometric consequences of the totally umbilical condition are established.

	Let $S$ be a totally umbilical submanifold of a Kenmotsu manifold $\bar M$. Then the following relation holds.
	\begin{equation}\label{e54}
		h(Z,W)=g(Z,W)\mathcal{H},
	\end{equation}
	\begin{equation}\label{e55}
		\mathcal{A}_{\beta}Z=g(\mathcal{H},\beta)Z
	\end{equation}for any $Z,W\in\Gamma(TS)$ and $\beta\in\Gamma(TS^\perp)$, where $\mathcal{H}$ denotes the mean curvature vector of $S$ in $\bar M$.
		
\begin{proposition}\label{P17}
Let $S$ be a totally umbilical PAInv-submanifold of a Kenmotsu manifold $\bar M$. Then the following conditions holds.
\begin{itemize}
	\item [(i)] The ambiguous distribution $\mathfrak{D}\oplus\langle\xi\rangle$ is integrable and leaves are totally geodesic in $S$ if and only if $(N\mathfrak{D}\oplus\mu)$ is parallel in normal bundle.
	 \item [(ii)] The leaves of anti-invariant distribution $\mathfrak{D^\perp}$ are totally geodesic
	 in $S$ if and only if $(\phi \mathfrak{D^\perp}\oplus\mu)$ is parallel in normal bundle.
\end{itemize}
\end{proposition}	

\begin{proof}
Let us Consider for any $W\in\Gamma(TS), \beta\in\Gamma(N\mathfrak{D}\oplus\mu)$, and using \eqref{e4} and \eqref{e17}, we get
\begin{equation*}
	g(\phi W,\beta)\xi-\eta(\beta)\phi W=\bar\nabla_W\phi \beta-\phi\bar\nabla_W\beta.
\end{equation*}
Since the Reeb vector field $\xi$ is tangent to $TS$, the left-hand side vanishes. Moreover, by applying \eqref{e6} and \eqref{e7}, we obtain
\begin{equation*}
	\phi(-\mathcal{A}_{\beta}W+\nabla_W^\perp\beta)=\bar\nabla_Wt\beta+h(W,t\beta)-\mathcal{A}_{f\beta}W+\nabla_W^\perp f\beta.
\end{equation*}
By comparing the tangential components on both sides, we obtain
\begin{equation}\label{e56}
	-T\mathcal{A}_{\beta}W+t\nabla_W^\perp \beta=\nabla_Wt\beta-\mathcal{A}_{f\beta}W.
\end{equation}As $S$ is totally umbilical, equations \eqref{e55} and \eqref{e56} imply that
\begin{equation*}
	\nabla_Wt\beta=t\nabla_W^\perp \beta-g(\mathcal{H},\beta)TW+g(\mathcal{H},f\beta)W.
\end{equation*}Taking inner product with $Z\in\Gamma(TS)$, we get
\begin{equation}\label{e57}
	g(\nabla_Wt\beta,Z)=g(\mathcal{H},f\beta)g(Z,W)-g(\mathcal{H},\beta)g(Z,TW)-g(\nabla_W^\perp \beta,\phi Z).
\end{equation}By using \eqref{e9}, we obtain
\begin{equation*}
	g(\phi\beta,\mathfrak{D^\perp})=g(t\beta,\mathfrak{D^\perp})+g(f\beta,\mathfrak{D^\perp}),
\end{equation*}which implies that $g(t\beta,\mathfrak{D^\perp})=0, i.e., t\beta\in\Gamma(\mathfrak{D}\oplus\langle\xi\rangle)$. Now, by virtue of \eqref{e57}, for any ${Y_1}\in\Gamma(\mathfrak{D}\oplus\langle\xi\rangle), \beta\in\Gamma(\mu)$ and ${X_1}\in\Gamma(\mathfrak{D^\perp})$, we obtain
\begin{equation*}
g(\nabla_{{Y_1}}t\beta,{X_1})=g(\mathcal{H},f\beta)g({Y_1},{X_1})-g(\mathcal{H},\beta)g({X_1},T{Y_1})-g(\nabla_{{Y_1}}^\perp \beta,\phi {X_1}),
\end{equation*}
which reduces to
\begin{equation*}
g(\nabla_{{Y_1}}t\beta,{X_1})=-g(\nabla_{{Y_1}}^\perp \beta,\phi {X_1}).
\end{equation*} If $\mathfrak{D}\oplus\langle\xi\rangle$ is totally geodesic, then $\nabla_{{Y_1}}t\beta\in\Gamma(\mathfrak{D}\oplus\langle\xi\rangle)$, so
\begin{equation*}
g(\nabla_{{Y_1}}^\perp \beta,\phi {X_1})=0.
\end{equation*}
Hence, the $\phi \mathfrak{D^\perp}$-component of $\nabla_{{Y_1}}^{\perp}\beta$ vanishes. Consequently,
$\nabla_{{Y_1}}^{\perp}\beta\in\Gamma(N\mathfrak{D}\oplus\mu),$
which shows that $N\mathfrak{D}\oplus\mu$ is parallel with respect to the normal connection. Conversely, reversing the above arguments, we conclude that $\mathfrak{D}\oplus\langle\xi\rangle$ is totally geodesic. The proof of part $(ii)$ follows by a similar argument.
\end{proof}	

\begin{proposition}\label{P18}
Let $S$ be a proper totally umbilical PAInv-submanifold of a Kenmotsu manifold $\bar M$. If ambiguous distribution $\mathfrak{D}\oplus\langle\xi\rangle$ is integrable and leaves are totally geodesic in $N$, then $\mathcal{H}\in\Gamma(N\mathfrak{D}\oplus\mu),$ where $\mathcal{H}$ denotes the mean curvature vector field of $S$.
\end{proposition}

\begin{proof}
For any ${Y_1}\in\Gamma(\mathfrak{D})$, we may write
\begin{equation*}
	g({Y_1},{Y_1})g(\mathcal{H},\phi \mathfrak{D^\perp})
	=
	g\bigl(g({Y_1},{Y_1})\mathcal{H},\phi \mathfrak{D^\perp}\bigr).
\end{equation*}
From \eqref{e54}, we get
\begin{equation*}
	g({Y_1},{Y_1})g(\mathcal{H},\phi \mathfrak{D^\perp})
	=
	g\bigl(h({Y_1},{Y_1}),\phi \mathfrak{D^\perp}\bigr).
\end{equation*}By virtue of \eqref{e4}, \eqref{e6} and \eqref{e9}, we have
\begin{equation*}
	||{Y_1}||^2g(\mathcal{H},\phi \mathfrak{D^\perp})=-g(\nabla_{{Y_1}}T{Y_1},\mathfrak{D^\perp})-g(\mathcal{A}_{N{Y_1}}{Y_1},\mathfrak{D^\perp})+g(g(\phi {Y_1},{Y_1})\xi-\eta({Y_1})\phi {Y_1},\mathfrak{D^\perp}).
\end{equation*} In the light of \eqref{e55} with fact that $\xi\in\Gamma(\mathfrak{D}\oplus\langle\xi\rangle)$, we get
\begin{equation*}
	||{Y_1}||^2g(\mathcal{H},\phi \mathfrak{D^\perp})=-g(\nabla_{{Y_1}}T{Y_1},\mathfrak{D^\perp})-g(\mathcal{H},N{Y_1})g({Y_1},\mathfrak{D^\perp}).
\end{equation*}Since the involved vector fields are mutually orthogonal, the second term on the right-hand side vanish. Therefore, we obtain
	\begin{equation*}
		||{Y_1}||^2g(\mathcal{H},\phi \mathfrak{D^\perp})=-g(\nabla_{{Y_1}}T{Y_1},\mathfrak{D^\perp}).
	\end{equation*}
	Furthermore, if the leaves of the ambiguous distribution are totally geodesic, then
	$\nabla_{{Y_1}}T{Y_1}\in\Gamma(\mathfrak{D}\oplus\langle\xi\rangle),$
	which implies $g(\nabla_{{Y_1}}T{Y_1},\mathfrak{D^\perp})=0.$
Hence,
	\begin{equation*}
		||{Y_1}||^2g(\mathcal{H},\phi \mathfrak{D^\perp})=0,
	\end{equation*}
	which completes the proof.
\end{proof}

\begin{proposition}\label{P19}
	Let $N$ be a proper totally umbilical PAInv-submanifold of a Kenmotsu manifold $\bar M$. If the distribution $\mathfrak{D^\perp}$ is integrable and its leaves are totally geodesic in $N$, then
	$\mathcal{H}\in\Gamma(\phi \mathfrak{D^\perp}\oplus\mu),$
	where $\mathcal{H}$ denotes the mean curvature vector field of $N$.
\end{proposition}

\begin{proof}
	Assume that $N$ is a proper submanifold, and let $0\neq {X_1}\in\Gamma(\mathfrak{D^\perp})$. Then,
	\begin{equation*}
		g({X_1},{X_1})g(\mathcal{H},N\mathfrak{D})
		=
		g(g({X_1},{X_1})\mathcal{H},N\mathfrak{D}).
	\end{equation*}
	By using \eqref{e54}, we obtain
	\begin{equation*}
		g({X_1},{X_1})g(\mathcal{H},N \mathfrak{D})
		=
		g\bigl(h({X_1},{X_1}),N\mathfrak{D}).
	\end{equation*}By virtue of \eqref{e4}, \eqref{e6} and \eqref{e9}, we get
	\begin{equation*}
		||{X_1}||^2g(\mathcal{H},N \mathfrak{D})=-g(\nabla_{{X_1}}{X_1},T\mathfrak{D})-g(\mathcal{A}_{\phi {X_1}}{X_1},\mathfrak{D})+g(g(\phi {X_1},{X_1})\xi-\eta({X_1})\phi {X_1},\mathfrak{D}).
	\end{equation*}In the light of \eqref{e55}, we have
	\begin{equation*}
		||{X_1}||^2g(\mathcal{H},N\mathfrak{D})=-g(\nabla_{{X_1}}{X_1},T\mathfrak{D})-g(\mathcal{H},N{X_1})g({X_1},\mathfrak{D}).
	\end{equation*}The second and term on the right-hand side vanish due to the orthogonality of the corresponding vector fields. Hence, we obtain
	\begin{equation*}
		||{X_1}||^2g(\mathcal{H},N\mathfrak{D})=-g(\nabla_{{X_1}}{X_1},T\mathfrak{D}).
	\end{equation*}
	Moreover, if the leaves of the distribution $\mathfrak{D^\perp}$ are totally geodesic, then
	\begin{equation*}
		\nabla_{{X_1}}T{X_1}\in\Gamma(\mathfrak{D^\perp}),
	\end{equation*}
	which implies
	\begin{equation*}
		g(\nabla_{{X_1}}T{X_1},\mathfrak{D})=0.
	\end{equation*}
	Hence,
	\begin{equation*}
		\|{X_1}\|^{2}g(\mathcal{H},N\mathfrak{D})=0,
	\end{equation*}
	from which the desired result follows.
\end{proof}

\begin{lemma}
	Let $S$ be a totally umbilical PAInv-submanifold of a Kenmotsu manifold $\bar M$. Then
\begin{equation}\label{e61}
	||T{Y_1}||^2\big\{\tilde{K}({X_1},\phi {Y_1})\tilde{K}(T{Y_1},{X_1})\}=||N{Y_1}||^2\big\{\tilde{K}(N{Y_1},{X_1})-\tilde{K}({X_1},\phi {Y_1})\},
\end{equation}for any ${Y_1}\in\Gamma(\mathfrak{D}\oplus\langle\xi\rangle)$ and ${X_1}\in\Gamma(\mathfrak{D^\perp})$.
\end{lemma}

\begin{proof}
By using \eqref{e12} and \eqref{e14}, for any
${Y_1}\in\Gamma(\mathfrak{D}\oplus\langle\xi\rangle)$ and ${X_1}\in\Gamma(\mathfrak{D^\perp})$, we have
\begin{equation}\label{e58}
	\begin{split}
		\tilde{R}(T{Y_1},{X_1};{Y_1},\phi {X_1})&=g(\nabla_{T{Y_1}}^\perp h({X_1},{Y_1})-h(\nabla_{T{Y_1}}{X_1},{Y_1})-h({X_1},\nabla_{T{Y_1}}{Y_1}),\phi {X_1})\\
		&~~~-g(\nabla_{{X_1}}^\perp h(T{Y_1},{Y_1})+h(\nabla_{{X_1}}T{Y_1},{Y_1})+h(T{Y_1},\nabla_{{X_1}}{Y_1}),\phi {X_1}).
		\end{split}
\end{equation} By using the total umbilicality of $S$ from\eqref{e54}, we obtain
\begin{equation*}\begin{split}
		\tilde{R}(T{Y_1},{X_1};{Y_1},\phi {X_1})&=g({Y_1},{X_1})g(\nabla_{T{Y_1}}^\perp \mathcal{H},\phi {X_1})-g({X_1},\nabla_{T{Y_1}}{Y_1})g(\phi {X_1},\mathcal{H})\\
		&~~~-g({Y_1},\nabla_{T{Y_1}}{X_1})g(\phi {X_1},\mathcal{H})-g(T{Y_1},{X_1})g(\nabla_{{X_1}}^\perp \mathcal{H},\phi {X_1})\\
		&~~~+g({Y_1},\nabla_{{X_1}}T{Y_1})g(\phi {X_1},\mathcal{H})+g(T{Y_1},\nabla_{{X_1}}{Y_1})g(\phi {X_1},\mathcal{H}).
	\end{split}
\end{equation*} We obtain $g({X_1},\nabla_{T{Y_1}}{Y_1})=-g({Y_1},\nabla_{T{Y_1}}{X_1}),$
and consequently, the second and third terms cancel each other. By the same argument, the fifth and sixth terms also cancel. Moreover, by the orthogonality of the vector fields, we have

\begin{equation*}
\tilde{R}(T{Y_1},{X_1},{Y_1},\phi {X_1})=0.
\end{equation*}It follows from the symmetry properties of the curvature tensor and \eqref{e9} that
\begin{equation}\label{e59}
	\tilde{R}(T{Y_1},{X_1};{X_1},T{Y_1})=-\tilde{R}(T{Y_1},{X_1};{X_1},N{Y_1}).
\end{equation}In view of the definition of the $\phi$-bisectional curvature in \eqref{e15}, equation \eqref{e59} becomes
\begin{equation}\label{e60}
	\tilde{R}(T{Y_1};{X_1},{X_1},N{Y_1})=-||T{Y_1}||^2||{X_1}||^2\tilde{K}(T{Y_1},{X_1}).
\end{equation}Moreover, the symmetry properties of the Riemannian curvature tensor imply that
\begin{equation*}\begin{split}
	0&=\tilde{R}(T{Y_1},{X_1};{X_1},\phi {Y_1})=\tilde{R}({X_1},\phi {Y_1};T{Y_1},{X_1})\\
	&=-\tilde{R}({X_1},\phi {Y_1};{X_1},T{Y_1})\\
	&=-\tilde{R}({X_1},\phi {Y_1};{X_1},\phi {Y_1})+\tilde{R}({X_1},\phi {Y_1};{X_1},N{Y_1})
\end{split}\end{equation*}
by virtue of \eqref{e9} together with the symmetry properties, we may write
\begin{equation*}
\begin{split}
	0&=\tilde{R}({X_1},\phi {Y_1};\phi {Y_1},{X_1})+\tilde{R}({X_1},T{Y_1};{X_1},N{Y_1})+\tilde{R}({X_1},N{Y_1};{X_1},N{Y_1})\\
	&=\tilde{R}({X_1},\phi {Y_1};\phi {Y_1},{X_1})-\tilde{R}(T{Y_1},{X_1};{X_1},N{Y_1})-\tilde{R}({X_1},N{Y_1};N{Y_1},{X_1}).
	\end{split}
\end{equation*}From \eqref{e60} and \eqref{e15}, we get
\begin{equation*}
	||\phi {Y_1}||^2||{X_1}||^2\tilde{K}({X_1},\phi {Y_1})+||T{Y_1}||^2||{X_1}||^2\tilde{K}(T{Y_1},{X_1})-||N{Y_1}||^2||{X_1}||^2\tilde{K}(N{Y_1},{X_1})=0.
\end{equation*}Finally, from \eqref{e9}, we get
\begin{equation*}
	||T{Y_1}||^2\big\{\tilde{K}({X_1},\phi {Y_1})\tilde{K}(T{Y_1},{X_1})\}=||N{Y_1}||^2\big\{\tilde{K}(N{Y_1},{X_1})-\tilde{K}({X_1},\phi {Y_1})\}.
\end{equation*}
\end{proof}

\begin{theorem}\label{T17}
		Let $S$ be a totally umbilical PAInv-submanifold of a Kenmotsu space form $\bar M(c)$. Then either $S$ is totally geodesic totally real submanifold or $c=-1$.
\end{theorem}

\begin{proof}
We now derive the $\phi$-bi-sectional curvature. Since
$\eta({Y_1})=0$ and $\eta({X_1})=0$ for every
${Y_1}\in\Gamma(\mathfrak{D})$ and ${X_1}\in\Gamma(\mathfrak{D^\perp})$,
the definition of the $\phi$-bi-sectional curvature yields
\begin{equation*}
	\tilde{K}({X_1},\phi {Y_1})
	=\frac{\tilde{R}({X_1},\phi {X_1},\phi^{2}{Y_1},\phi {Y_1})}
	{\|{X_1}\|^{2}\|\phi {Y_1}\|^{2}}.
\end{equation*}

From \eqref{e1}, we get

\begin{equation*}
	\tilde{R}( {X_1},\phi {X_1},\phi^2 {Y_1},\phi {Y_1})
	=-\tilde{R}( {X_1},\phi {X_1}, {Y_1},\phi {Y_1})
	+\tilde{R}( {X_1},\phi {X_1},\xi,\phi {Y_1}).
\end{equation*}

In the light of \eqref{e10} together with $\eta({Y_1})=\eta({X_1})=0$, $\tilde{R}( {X_1},\phi {X_1},\xi,\phi {Y_1})=0$. Then, we have

\begin{equation*}
	\tilde{R}( {X_1},\phi {X_1},\phi^2 {Y_1},\phi {Y_1})
	=-\tilde{R}( {X_1},\phi {X_1}, {Y_1},\phi {Y_1}).
\end{equation*}

Now, by virtue of \eqref{e10}, we get

\begin{equation*}
	\tilde{R}( {X_1},\phi {X_1},\phi^2 {Y_1},\phi {Y_1})
	=\frac{(c+1)}{2}{|| {X_1}||^2||\phi {Y_1}||^2}.
\end{equation*}

By the definition of the $\phi$-bi-sectional curvature, we obtain
$\tilde{K}({X_1},\phi {Y_1})=\frac{(c+1)}{2}.$
Similarly, $\tilde{K}({X_1},N{Y_1})=\frac{(c+1)}{2}$
\quad\text{and}
$\tilde{K}({X_1},T{Y_1})=\frac{(c+1)}{2}.$
Substituting these expressions into \eqref{e61}, we obtain
$\|T{Y_1}\|^2(c+1)=0.$
Consequently, either $T{Y_1}=0$, which implies that $S$ is totally real, or $c=-1$. This completes the proof.
\end{proof}

\begin{theorem}\label{T23}
	Let $S$ be a totally umbilical PAInv-submanifold of a Kenmotsu space form $\bar M$. Then at least one of the following statement holds.
\begin{itemize}
	\item [(i)] $S$ is totally geodesic submanifold of $\bar M$.
	\item [(ii)] The mean curvature vector satisfies that $\mathcal{H}\in\Gamma(N\mathfrak{D}\oplus\mu)$.
\item [(iii)] If $\mathcal{H}\in\Gamma(\phi \mathfrak{D^\perp})$, then dim$(\mathfrak{D^\perp})=1$.
\end{itemize}
\end{theorem}

\begin{proof}
By virtue of Lemma \ref{L3}, we get
$g(h({X_1},W),N{X_2})=g(h({X_2},W),N{X_1})$, for any
${X_1},{X_2},W\in\Gamma(\mathfrak{D^\perp})$. Since $S$ is totally umbilical, it follows that
\begin{equation*}
	g({X_1},W)g(\mathcal{H},N{X_2})=g({X_2},W)g(\mathcal{H},N{X_1}).
\end{equation*}
Considering  ${X_1}=W$ in the above equation, we get
\begin{equation}\label{e64}
	g(W,W)g(\mathcal{H},\phi {X_2})=g({X_2},W)g(\mathcal{H},\phi W).
\end{equation}
Interchanging the roles of ${X_2}$ and $W$ in \eqref{e64} and combining the resulting equation with \eqref{e64}, we obtain
\begin{equation}\label{e67}
	||{X_2}||^2g(\mathcal{H},\phi W)=\frac{g({X_2},W)^2g(\mathcal{H},\phi W)}{||W||^2}.
\end{equation}
 By the Cauchy--Schwartz inequality,
 $g({X_2},W)^2\leq ||{X_2}||^2||W||^2$, which implies that
 $\frac{g({X_2},W)^2}{||{X_2}||^2||W||^2}\leq 1.$
 Thus, \eqref{e67} can be rewritten as
 $g(\mathcal{H},\phi W)=\lambda g(\mathcal{H},\phi W),$
 where
 $\lambda=\frac{g({X_2},W)^2}{||{X_2}||^2||W||^2}\leq 1.$
 
We now consider two cases. If $\mathcal{H}=0$, then
$h({X_1},{X_2})=g({X_1},{X_2})\mathcal{H}=0,$
which implies that $S$ is a totally geodesic submanifold. This completes the proof of part~$(i)$.
 
 On the other hand, suppose that $\mathcal{H}\neq 0$. Then
 \begin{equation*}
 	g(\mathcal{H},\phi W)=\lambda g(\mathcal{H},\phi W).
 \end{equation*}
 To prove part~$(ii)$, assume that $\lambda=0$. Then
 \begin{equation*}
 	g(\mathcal{H},\phi W)=0
 \end{equation*}
 for every $W\in\Gamma(\mathfrak{D^\perp})$. Hence, $\mathcal{H}$ is orthogonal to $\phi \mathfrak{D^\perp}$, and consequently,
 \begin{equation*}
 	\mathcal{H}\in\Gamma(N\mathfrak{D}\oplus\mu).
 \end{equation*}
 
 If $\lambda=1$, then
 $g({X_2},W)^2=||{X_2}||^2||W||^2$.
 By the equality case of the Cauchy--Schwartz inequality, ${X_2}$ and $W$ are linearly dependent, i.e., ${X_2}=aW$ for some scalar $a$. Since ${X_2}$ and $W$ are arbitrary in $\mathfrak{D^\perp}$, this is possible only when $\dim(\mathfrak{D^\perp})=1$. Hence, the proof is complete.
\end{proof}
	
\begin{theorem}\label{T24}
Let $S$ be a PAInv- submanifold of a Kenmotsu space form $\bar M(c)$. Suppose that the morphism $N$ is parallel, i.e., $\bar{\nabla}N=0$, and that the distributions $\mathfrak{D}$ and $\mathfrak{D^\perp}$ are integrable whose leaves are totally geodesic in $S$. Then $c\geq 1$. Furthermore, if $c=-1$, then
$fh(\mathfrak{D},\mathfrak{D^\perp})=0$.
\end{theorem}

\begin{proof}
Let $S$ be a PAInv-submanifold of a Kenmotsu space form $\bar{M}(c)$ such that $\mathfrak{D}$ and $\mathfrak{D^\perp}$ define totally geodesic foliations on $S$ and satisfy $\eta(\mathfrak{D})=\eta(\mathfrak{D^\perp})=0$. If $F_{\mathfrak{D}}$ and $F_{\mathfrak{D^\perp}}$ denote the integral manifolds of $\mathfrak{D}$ and $\mathfrak{D^\perp}$, respectively, then for any ${Y_1},{Y_2}\in\Gamma(F_{\mathfrak{D}})$ and ${X_1}\in\Gamma(F_{\mathfrak{D^\perp}})$, we have
$$
g(\nabla_{ {Y_1}}{Y_2}, {X_1})
=-g(\nabla_{ {Y_1}} {X_1},{Y_2})=0,
$$
which implies that
$\nabla_{ {Y_1}} {X_1}\in\Gamma(F_{\mathfrak{D^\perp}})$.
Similarly, one can show that
$\nabla_{ {X_1}} {Y_1}\in\Gamma(F_{\mathfrak{D}})$. Then, we get
\begin{equation}\label{e68}
	\nabla_{ {Y_1}} {X_1}\in\Gamma(F_{\mathfrak{D^\perp}}),\qquad \nabla_{ {X_1}} {Y_1}\in\Gamma(F_{\mathfrak{D}}).
\end{equation}
Since $N$ is parallel, it follows that $\bar\nabla_{ {X_1}}N {Y_1}=0$, which implies that $\nabla_{ {X_1}}^\perp N {Y_1}\in\Gamma(NF_{\mathfrak{D}})$. By using \eqref{e7}, \eqref{e9} and \eqref{e11}, we obtain
\begin{equation*}
	0=g(\nabla_{ {X_1}}^\perp N {Y_1},\phi {X_1})=g(\phi\bar\nabla_{ {X_1}} {Y_1},\phi {X_1})-g((\bar\nabla_{ {X_1}}\phi) {Y_1},\phi {X_1})+g(h( {X_1},T {Y_1}),\phi {X_1}).
\end{equation*}
From \eqref{e4} and \eqref{e1}, we have
\begin{equation*}
	g(\nabla_{ {X_1}} {Y_1}, {X_1})+g(h( {X_1},T {Y_1}),\phi {X_1})=0.
\end{equation*}
In the light of \eqref{e68}, we have 
\begin{equation}\label{e69}
	g(h( {X_1},T {Y_1}),\phi {X_1})=0.
\end{equation}By virtue of \eqref{e14} and \eqref{e12}, for any ${X_2}\in\Gamma(\mathfrak{D^\perp})$, we get
\begin{equation}\label{e70}
	\begin{split}
	g(\tilde{R}( {Y_1},{Y_2}) {X_1},\phi {X_2})&=g(\nabla_{ {Y_1}}^\perp h({Y_2}, {X_1}),\phi {X_2})-g(h(\nabla_{ {Y_1}}{Y_2}, {X_1}),\phi {X_2})\\
	&~~~-g((h(\nabla_{ {Y_1}} {X_1},{Y_2}),\phi {X_2})-g(\nabla_{{Y_2}}^\perp h( {Y_1}, {X_1}),\phi {X_2})\\
	&~~~+g(h(\nabla_{{Y_2}} {Y_1}, {X_1}),\phi {X_2})+g((h(\nabla_{{Y_2}} {X_1}, {Y_1}),\phi {X_2}). 
	\end{split}
\end{equation}
From \eqref{e3}, \eqref{e4}, \eqref{e9} with \eqref{e11}, we have $g(h(\mathfrak{D},\mathfrak{D^\perp}),\phi \mathfrak{D^\perp})=0$. By using this fact
with \eqref{e6}, \eqref{e11} and \eqref{e68}, we have
\begin{equation*}
\begin{split}
	g(\tilde{R}( {Y_1},{Y_2}) {X_1},\phi {X_2})&=g(h( {Y_1}, {X_1}),\phi\nabla_{{Y_2}}{X_2})-g(h( {Y_1}, {X_1}),(\bar\nabla_{{Y_2}}\phi){X_2})\\
	&~~~-g(h({Y_2}, {X_1}),\phi\nabla_{ {Y_1}}{X_2})+g(h({Y_2}, {X_1}),(\bar\nabla_{ {Y_1}}\phi){X_2}).
\end{split}
\end{equation*}
Considering \eqref{e4} to the above equation, we get
\begin{equation*}
\begin{split}
	g(\tilde{R}( {Y_1},{Y_2}) {X_1},\phi {X_2})&=g(h( {Y_1}, {X_1}),\phi\bar\nabla_{{Y_2}}{X_2})-g(\phi {Y_2},{X_2})g(h( {Y_1}, {X_1}),\xi)\\
	&~~~+\eta({X_2})g(h( {Y_1}, {X_1}),\phi {Y_2})-g(h({Y_2}, {X_1}),\phi\bar\nabla_{ {Y_1}}{X_2})\\
	&~~~+g(\phi {Y_1},{X_2})g(h({Y_2}, {X_1}),\xi)-\eta({X_2})g(h({Y_2}, {X_1}), \phi {Y_1}).
\end{split}
\end{equation*}
From \eqref{e3} and \eqref{e6}, together with the orthogonality of the vector fields, and setting ${Y_2}=T{Y_1}$ and ${X_2}={X_1}$, we obtain
\begin{equation*}
g(\tilde{R}( {Y_1},T{Y_1}) {X_1},\phi {X_1})=-2g(\phi h( {Y_1}, {X_1}),h(T {Y_1}, {X_1})).
\end{equation*}By virtue of \eqref{e34}, we get
\begin{equation}\label{e73}
	g(\tilde{R}( {Y_1},T{Y_1}) {X_1},\phi {X_1})=-2||fh( {Y_1}, {X_1})||^2.
\end{equation}From \eqref{e10}, the above equation reduces to
\begin{equation}\label{e74}
\frac{(c+1)}{2}||T {Y_1}||^2|| {X_1}||^2=2||fh( {Y_1}, {X_1})||^2.
\end{equation}
By virtue of \eqref{e74}, it follows that
\begin{equation*}
\frac{(c+1)}{2}||T {Y_1}||^2|| {X_1}||^2\geq 0.
\end{equation*}
Since $S$ is a proper PAInv-submanifold, it implies that
\begin{equation*}
\frac{(c+1)}{2}\geq 0,
\end{equation*}
which implies that $c\geq -1$. Furthermore, if $c=-1$, then \eqref{e74}, we get
\begin{equation*}
fh(\mathfrak{D},\mathfrak{D^\perp})=0.
\end{equation*}
Hence, the proof is complete.
\end{proof}

We are now in a position to derive an inequality involving the squared norm of the second fundamental form of a proper PAInv-submanifold in a Kenmotsu space form. The proof relies essentially on the result established above.

\begin{theorem}\label{T25}

	Let $S$ be a proper PAInv-submanifold of a Kenmotsu space form $\bar M(c)$, and suppose that the endomorphism $N$ is parallel. If the distributions $\mathfrak{D}$ and $\mathfrak{D^\perp}$ are totally geodesic in $S$, then the squared norm of the second fundamental form satisfies
	\begin{equation*}
		||h||^2\geq \frac{s(c+1)}{2}||T||^2,
	\end{equation*}
	where $s=\dim(\mathfrak{D^\perp})$ and
	\begin{equation*}
		||T||^2=\sum\limits_{i,j=1}^{n}g(Te_i,e_j)^2,
	\end{equation*}
	with $\{e_1,e_2,\ldots,e_s\}$ and $\{\tilde{e}_1,\tilde{e}_2,\ldots,\tilde{e}_r\}$ denoting orthonormal frames of $\mathfrak{D^\perp}$ and $\mathfrak{D}$, respectively. Moreover, if equality holds, then the integral manifolds $F_{\mathfrak{D}}$ and $F_{\mathfrak{D^\perp}}$ are totally geodesic in $\bar M(c)$. Furthermore, $F_{\mathfrak{D^\perp}}$ is a Kenmotsu space form of constant sectional curvature $\frac{c-3}{4}$, whereas the integral manifold $F_{\mathfrak{D}}$ has sectional curvature
	\begin{equation*}\begin{split}
		R^{F_{\mathfrak{D}}}({X_1},{X_2})Z&=\frac{(c-3)}{4}\big[g({X_2},Z){X_1}-g({X_1},Z){X_2}\big]\\
		&
		+\frac{(c+1)}{4}\big[g(T{X_2},Z)T{X_1}-g(T{X_1},Z)T{X_2}+2g(T{X_2},{X_1})TZ\big].
		\end{split}
	\end{equation*}
	Consequently, $F_{\mathfrak{D}}$ is either an invariant or an anti-invariant submanifold.
\end{theorem}

\begin{proof}
	By considering the orthonormal frame, we may write
	\begin{equation*}
		\begin{split}
		||h||^2&=\sum\limits_{i,j=1}^{n}||h(e_i,e_j)||^2\\
		&=\sum\limits_{i,j=1}^{s}||h(e_i,e_j)||^2+\sum\limits_{i,j=1}^{r+1}||h(\tilde{e_i},\tilde{e_j})||^2+2\sum\limits_{i=1}^{s}\sum\limits_{j=1}^{r+1}||h(e_i,\tilde{e_j})||^2.
		\end{split}
	\end{equation*}Since, in Kenmotsu manifold, there is a role of structure vector field $\xi$, we have
	
		\begin{equation*}
		\begin{split}
			||h||^2&=\sum\limits_{i,j=1}^{s}||h(e_i,e_j)||^2+\sum\limits_{i}^{r}||h(\tilde{e_i},\xi)||^2+\sum\limits_{i,j=1}^{r}||h(\tilde{e_i},\tilde{e_j})||^2+2\sum\limits_{i=1}^{s}\sum\limits_{j=1}^{r}||h(e_i,\tilde{e_j})||^2\\
			&~~~+2\sum\limits_{j=1}^{s}||h({e_j},\xi)||^2.
		\end{split}
	\end{equation*}By virtue of \eqref{e9}, we get
	
		\begin{equation}\label{e75}
		\begin{split}
			||h||^2 &=||h_\perp||^2+\sum\limits_{i}^{r}||h(\tilde{e_i},\xi)||^2+||h_c||^2+2\sum\limits_{i=1}^{s}\sum\limits_{j=1}^{r}g(th(e_i,\tilde{e_j}),th(e_i,\tilde{e_j})\\
			&~~~+2\sum\limits_{i=1}^{s}\sum\limits_{j=1}^{r}g(fh(e_i,\tilde{e_j}),fh(e_i,\tilde{e_j})+2\sum\limits_{j=1}^{s}||h({e_j},\xi)||^2.
		\end{split}
	\end{equation}
\end{proof}For any ${W}\in\Gamma(TS)$, by using \eqref{e5}, \eqref{e6}, and \eqref{e9}, we have
\begin{equation*}
h(W,\xi)=0,
\end{equation*}
which implies that the second and sixth terms on the right-hand side of \eqref{e75} vanish.
Substituting these expressions with \eqref{e74} into \eqref{e75}, we obtain
\begin{equation*}
	||h||^2\geq \frac{s(c+1)}{2}||T||^2.
\end{equation*}This completes  the proof of Theorem. If the equality holds, then from \eqref{e75}, we obtain
\begin{equation}\label{e81}
||h_\perp||^2=0 \qquad \text{and} \qquad ||h_c||^2=0.
\end{equation}
Hence, the integral manifolds $F_{\mathfrak{D^\perp}}$ and $F_{\mathfrak{D}}$ are totally geodesic in $\bar M(c)$. Since $F_{\mathfrak{D^\perp}}$ is totally geodesic, the Gauss equation yields
\begin{equation*}
\tilde{R}({X_1},{X_2};Z,W)
=
R({X_1},{X_2};Z,W)
+g(h({X_1},W),h({X_2},Z))
-g(h({X_1},Z),h({X_2},W)).
\end{equation*}
By using \eqref{e81}, we obtain
\[
\tilde{R}({X_1},{X_2};Z,W)=R({X_1},{X_2};Z,W),
\]
for any ${X_1},{X_2},Z,W\in\Gamma(\mathfrak{D^\perp})$. Therefore, by \eqref{e10} and the fact that $\eta(Z)=0$ for every $Z\in\Gamma(\mathfrak{D^\perp})$, we have
\begin{equation*}
R({X_1},{X_2})Z
=
\frac{c-3}{4}
\big\{
g({X_2},Z){X_1}-g({X_1},Z){X_2}
\big\}.
\end{equation*}
Consequently, $F_{\mathfrak{D^\perp}}$ is a real space form of constant sectional curvature $\frac{c-3}{4}$.

Since $F_{\mathfrak{D}}$ is also totally geodesic, it follows from \eqref{e10} and \eqref{e9} that
\begin{equation}\label{e80}
\begin{split}
	R({X_1},{X_2})Z
	&=
	\frac{c-3}{4}
	\big\{
	g({X_2},Z){X_1}-g({X_1},Z){X_2}
	\big\}\\
	&\quad
	+\frac{c+1}{4}
	\big\{
	g(T{X_2},Z)T{X_1}-g(T{X_1},Z)T{X_2}+2g(T{X_2},{X_1})TZ
	\big\},
\end{split}
\end{equation}
and
\begin{equation*}
g(T{X_2},Z)N{X_1}-g(T{X_1},Z)N{X_2}+2g(T{X_2},{X_1})NZ=0,
\end{equation*}
for any ${X_1},{X_2},Z\in\Gamma(\mathfrak{D})$.

Putting $Z={X_1}$, we obtain
\[
3g(T{X_2},{X_1})N{X_1}=0.
\]
Hence, either $N{X_1}=0$, in which case $F_{\mathfrak{D}}$ is an invariant submanifold, or $g(T{X_2},{X_1})=0$ for every ${X_1}\in\Gamma(\mathfrak{D})$. The latter implies that $T{X_2}\perp \mathfrak{D}$. On the other hand, by the definition of the ambiguous distribution, we have $T{X_2}\in\Gamma(\mathfrak{D})$, for every ${X_2}\in\Gamma(\mathfrak{D})$, from which, we have $T{X_2}=0$. It implies that,
\[
\phi {X_2}=N{X_2}\in \Gamma(TS^\perp),
\]
which shows that $F_{\mathfrak{D}}$ is an anti-invariant submanifold. This proves the inequality.

\section*{Declarations}
\begin{itemize}
	
	\item \textbf{Funding}\
	The authors received no financial support for the research, authorship, and/or publication of this article.\\
	
	\item \textbf{Conflict of interest/Competing interests}\
	The authors declare that they have no conflict of interest and no competing interests.\\
	
	\item \textbf{Ethics approval and consent to participate}\
	Not applicable.\\
	
	\item \textbf{Consent for publication}\
	Not applicable.\\
	
	\item \textbf{Data availability}\
	No datasets were generated or analyses during the current study.\\
	
	
	\item \textbf{Code availability}\
	Not Applicable.\\
	
	\item \textbf{Author contribution}\
	All authors contributed equally to the conception, development, analysis, and writing of this manuscript. All authors read and approved the final manuscript.
	
\end{itemize}

\end{document}